\documentclass[11pt]{amsart}
\pdfoutput=1

\usepackage{appendix}

\newcommand\WriteupNote[1]{}

\usepackage{comment}
\usepackage{caption}
\usepackage{subcaption}
\usepackage{graphicx}
\usepackage{enumerate,enumitem}
\usepackage{color}
\usepackage{mathtools}
\mathtoolsset{showonlyrefs}

\usepackage{fancyhdr}
\usepackage{epsfig,amsfonts,latexsym,amsmath,amssymb}
\usepackage{amsthm}

\theoremstyle{plain}
 \newtheorem{theorem}{Theorem}[section]
 \newtheorem{lemma}[theorem]{Lemma}
 \newtheorem{corollary}[theorem]{Corollary}
 \newtheorem{proposition}[theorem]{Proposition}
  \newtheorem{example}[theorem]{Example}

 \theoremstyle{definition}
 \newtheorem{definition}[theorem]{Definition}
 \theoremstyle{definition}
  \newtheorem{remark}[theorem]{Remark}

\newtheorem{conjecture}[theorem]{Conjecture}

\usepackage{hyperref}

\numberwithin{equation}{section}

\DeclareMathOperator {\Id} {Id}

\newcommand{\R}{\mathbb R}
\newcommand{\Q}{\mathbb Q}

\newcommand{\der}{\mathrm{d}}
\newcommand{\eps}{\varepsilon}
\renewcommand{\phi}{\varphi}
\newcommand{\abs}[1]{\left\lvert #1 \right\rvert}

\newcommand{\Der}[1]{\frac{\der}{\der #1}}
\DeclareMathOperator{\sgn}{sgn}

\newcommand{\lsp}{\operatorname{lsp}}
\newcommand{\spec}{\operatorname{spec}}

\newcommand{\Tr}{\operatorname{Tr}}

\newcommand{\pdpdo}[2]{\frac{{\mathrm{d}} #1}{{\mathrm{d}} #2}}

\newcommand{\dd}{\operatorname{d} \!}
\newcommand{\ii}{\operatorname{i}}
\newcommand{\RR}{\mathbb{R}}

\renewcommand{\Im}{\operatorname{Im}}
\renewcommand{\Re}{\operatorname{Re}}
\renewcommand{\det}{\operatorname{det}}

\newcommand{\mstrut}[1]{\mbox{\rule{0mm}{#1}}}
\newcommand{\N}{\mathbb{N}}

\newcommand{\ba}{\begin{array}}
\newcommand{\ea}{\end{array}}

\def\beq{\begin{equation} } \def\eeq{\end{equation}}

\def\RR{\mathbb R}

\def\eps{\varepsilon}  

\def\ben{\begin{enumerate} }

\def\een{\end{enumerate} }

\def \R{ {\mathbb R}}
  
\def \p { \partial}

\def \Id{ \operatorname{Id} } 

\def \surf{ 1}  \def \rturn{ R^\star}
\def \CMB{ R} \def \wkb{\beta}

\title[Spectral rigidity for spherically symmetric manifolds with
  boundary]{Spectral rigidity for spherically symmetric manifolds
  with boundary}

\author{Maarten V. de Hoop}
\thanks{Simons Chair in Computational and Applied Mathematics and Earth Science, Rice University, Houston TX, USA. \texttt{mdehoop@rice.edu}}
\author{Joonas Ilmavirta}
\thanks{Department of Mathematics and Statistics, University of Jyv\"askyl\"a, Finland. \texttt{joonas.ilmavirta@jyu.fi}}
\author{Vitaly Katsnelson}
\thanks{Department of Computational and Applied Mathematics, Rice University, Houston TX, USA. \texttt{vitaly.katsnelson@rice.edu}}

\begin{document}

\begin{abstract}
We prove a trace formula for three-dimensional spherically symmetric
Riemannian manifolds with boundary which satisfy the Herglotz
condition: The wave trace is singular precisely at the length spectrum
of periodic broken rays. In particular, the Neumann spectrum of the
Laplace--Beltrami operator uniquely determines the length spectrum.
The trace formula also applies for the toroidal modes of the free
oscillations in the earth. We then prove that the length spectrum is
rigid: Deformations preserving the length spectrum and spherical
symmetry are necessarily trivial in any dimension, provided the
Herglotz condition and a generic geometrical condition are satisfied.
Combining the two results shows that the Neumann spectrum of the
Laplace--Beltrami operator is rigid in this class of manifolds with
boundary.
\end{abstract}


\maketitle


\section{Introduction}
 We establish spectral
rigidity for spherically symmetric manifolds with boundary. We study
the recovery of a (radially symmetric Riemannian) metric or wave speed
rather than an obstacle. To our knowledge it is the first such result
pertaining to a manifold with boundary. We require the so-called
Herglotz condition while allowing an unsigned curvature; that is,
curvature can be everywhere positive or it can change sign, and we
allow for conjugate points. Spherically symmetric manifolds with
boundary are models for planets, the preliminary reference Earth model
(PREM) being the prime example. Specifically, restricting to toroidal
modes, our spectral rigidity result determines the shear wave speed of
Earth's mantle in the rigidity sense.

The method of proof relies on a trace formula, relating the spectrum
of the manifold with boundary to its length spectrum, and the
injectivity of the periodic broken ray transform. Specifically, our
manifold is the Euclidean annulus $M = \bar B(0,1) \setminus \bar
B(0,R) \subset \R^n$, $R>0$ and $n\geq2$, with the metric $g(x) =
c^{-2}(\abs{x}) e(x)$, where $e$ is the standard Euclidean metric and
$c \colon (R,1] \to (0,\infty)$ is a function satisfying suitable
  conditions. (In appendix~\ref{app:symmetry} we clarify that any
  spherically symmetric manifold is in fact of the form we consider --
  radially conformally Euclidean.) Our assumption is that the Herglotz
  condition $\Der{r}(r/c(r))>0$ is satisfied everywhere. This
  condition was first discovered by Herglotz~\cite{H:kinematic} and
  used by Wiechert and Zoeppritz~\cite{WZ:kinematic}.

By a maximal geodesic we mean a unit speed geodesic on the Riemannian
manifold $(M,g)$ with endpoints at the outer boundary $\partial
M\coloneqq\partial B(0,1)$. A broken ray or a billiard trajectory is a
concatenation of maximal geodesics satisfying the reflection condition
of geometrical optics at both inner and outer boundaries of $M$. If the initial and final points
of a broken ray coincide at the boundary, we call it a periodic broken ray -- in
general, we would have to require the reflection condition at the
endpoints as well, but in the assumed spherical symmetry it is
automatic. We will describe later (definition~\ref{def:ccc}) what will
be called the \emph{countable conjugacy condition} which ensures
that up to rotation only countably many maximal geodesics have conjugate endpoints.
This is a ``generic'' condition for reasonable wave speeds.

The length spectrum of a manifold $M$ with boundary is the set of
lengths of all periodic broken rays on $M$. If $M$ is a spherically
symmetric manifold as described above, we may choose whether or not we
include the rays that reflect on the inner boundary $r=R$. If the
radial sound speed is $c$, we denote the length spectrum without these
interior reflections by $\lsp(c)$ and the one with these reflections
by $\lsp'(c)$. See figure~\ref{fig:lsp} below for an illustration of these two
types of geodesics.
If the inner radius
is zero ($R=0$), the manifold is essentially a ball and the two kinds
of length spectra coincide.
We note that every broken ray is contained in a
unique two-dimensional plane in $\R^n$ due to symmetry
considerations. Therefore, it will suffice to consider the case $n=2$;
the results regarding geodesics and the length spectrum carry over to
higher dimensions.

\begin{figure}[ht]
\includegraphics[width=0.5\textwidth,scale = 0.8]{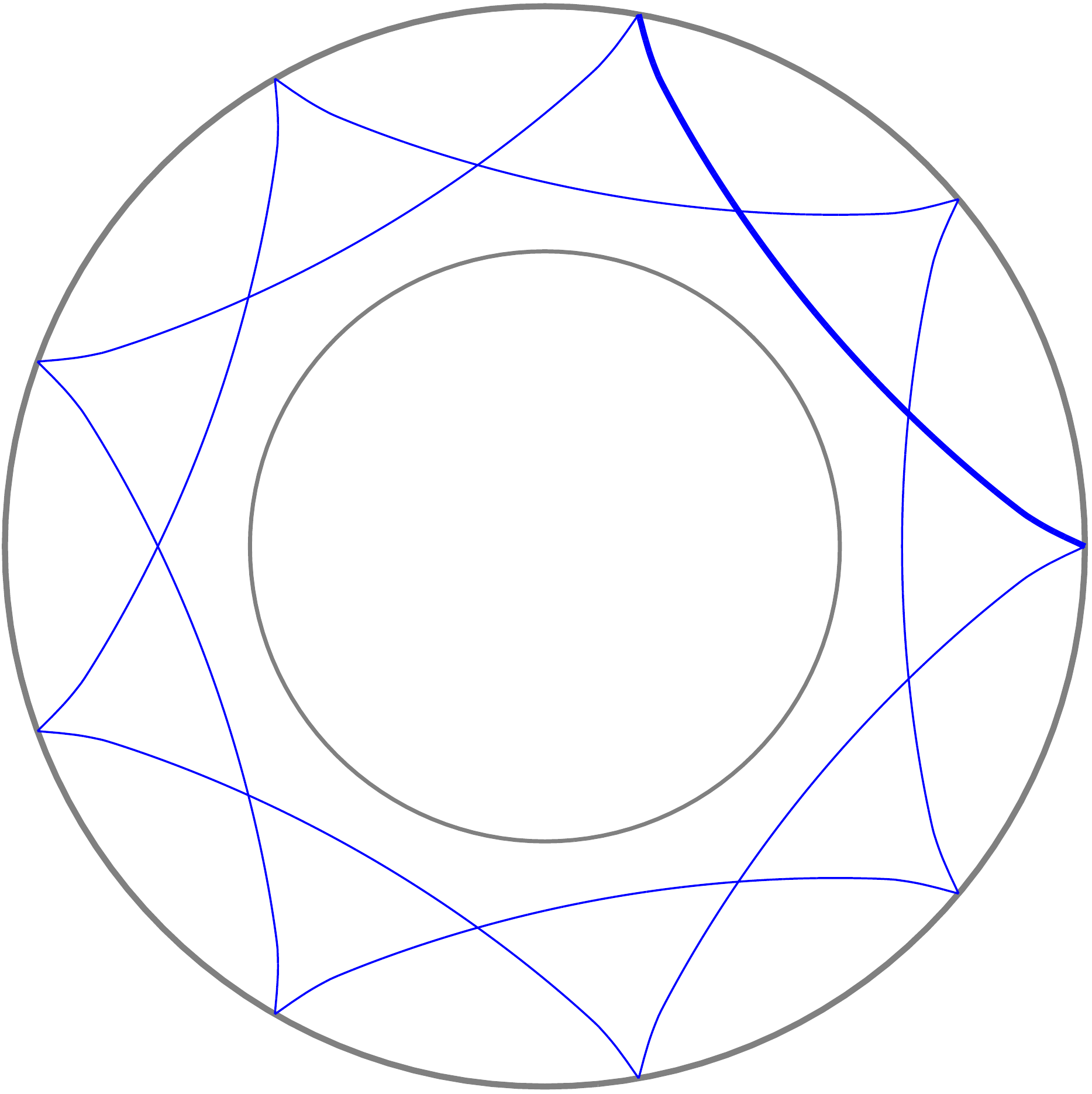}
\quad
\includegraphics[width=.5\textwidth, scale = 0.8]{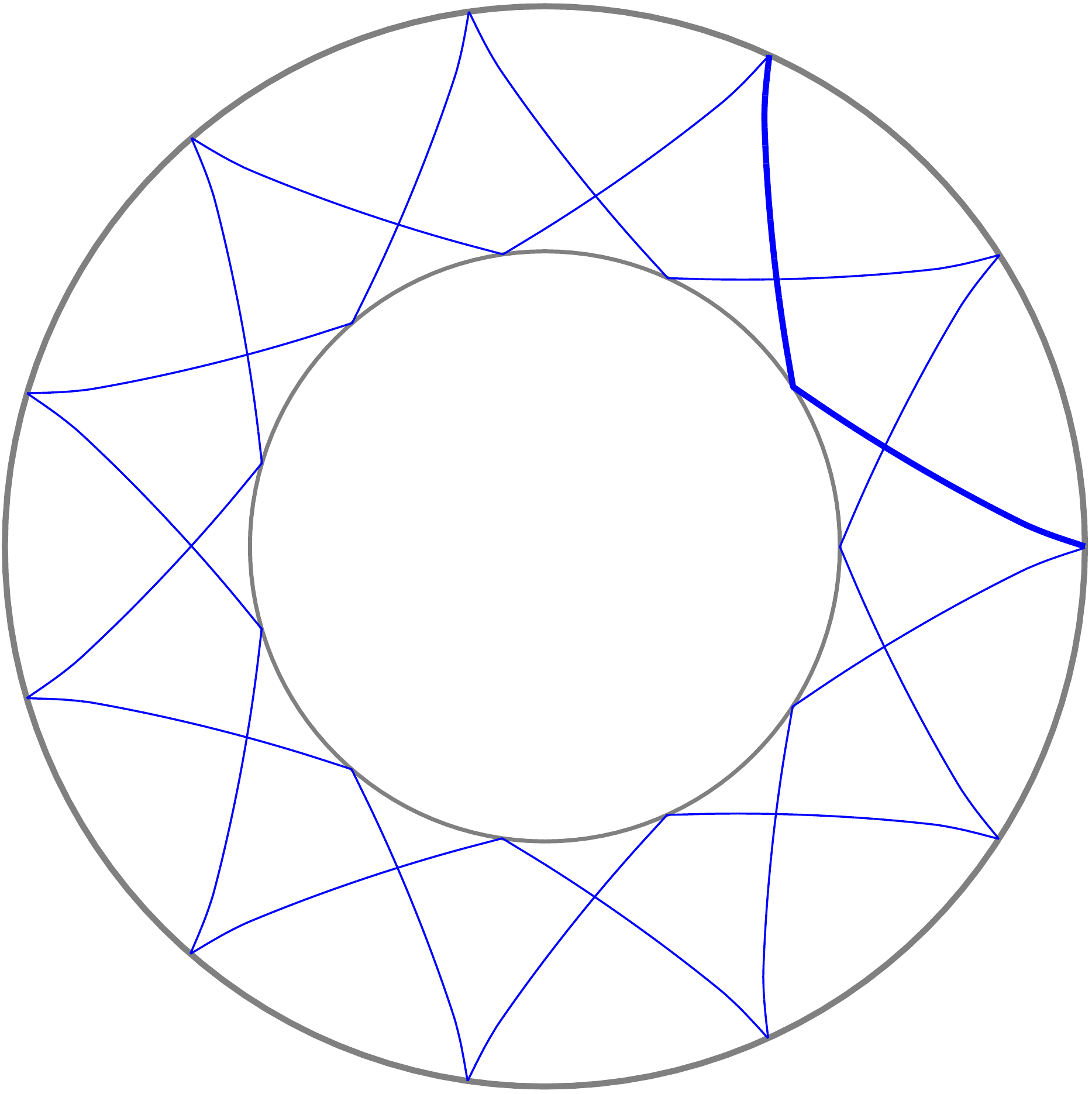}
\caption{The first figure shows a closed orbit whose period is in
  $\lsp(c)$ and $\lsp'(c)$. The second figure shows a closed orbit
  whose period is in $\lsp'(c)$, but not in $\lsp(c)$. }
\label{fig:lsp}
\end{figure}

We denote the Neumann spectrum of the Laplace--Beltrami operator in
three dimensions, $\Delta_c = c^3 \nabla \cdot c^{-1} \nabla$, on $M$
by $\spec(c)$, where we impose Neumann-type boundary
conditions on both the inner and outer boundary.
The spectrum $\spec(c)$ includes multiplicity, not just the set spectrum.

Some earlier results in tensor tomography the methods of which are
related to ours may be found in
\cite{Anasov14,Beurling15,Sharaf97,UhlSharaf}.
We are not aware of prior spectral rigidity results for manifolds with boundary.
Two of the main theorems we prove
on spectral rigidity are the following:

\begin{theorem} \label{thm:lsp-rig-intro}
Let $B=\bar B(0,1)\setminus \bar B(0,R)\subset\R^n$, $n\geq2$ and
$R\geq0$, be an annulus (or a ball if $R=0$).  Fix $\eps>0$ and let
$c_\tau$, $\tau\in(-\eps,\eps)$, be a $C^{1,1}$ function
$[R,1]\to(0,\infty)$ satisfying the Herglotz condition and the
countable conjugacy condition and depending $C^1$-smoothly on the
parameter $\tau$.  If $R=0$, we assume $c_\tau'(0)=0$.  If
$\lsp(c_\tau)=\lsp(c_0)$ for all $\tau\in(-\eps,\eps)$, then
$c_\tau=c_0$ for all $\tau\in(-\eps,\eps)$.

The result holds true also for the length spectrum $\lsp'(c)$
including reflections at the inner boundary, provided that $R>0$.
\end{theorem}

\begin{theorem} \label{thm: intro spec rigidity}
Let $B=\bar B(0,1)\setminus \bar B(0,R)\subset\R^3$, $R\geq0$, be an
annulus (or a ball if $R=0$).  Fix $\eps>0$ and let $c_\tau$,
$\tau\in(-\eps,\eps)$, be a $C^\infty$ function $[R,1]\to(0,\infty)$
satisfying the Herglotz condition and the countable conjugacy
condition and depending $C^1$-smoothly on the parameter $\tau$.
Assume also that the length spectrum is non-degenerate in the sense that if two periodic broken rays have the same primitive period, then they differ by a rotation.
If
$R=0$, we assume that all odd order derivatives of $c_\tau$ vanish at
$0$.  If $\spec(c_\tau)=\spec(c_0)$ for all $\tau\in(-\eps,\eps)$,
then $c_\tau=c_0$ for all $\tau\in(-\eps,\eps)$.
\end{theorem}

Theorem~\ref{thm:lsp-rig-intro} is restated as
theorems~\ref{thm:lsp-rig-diving} (for $\lsp(c)$)
and~\ref{thm:lsp-rig-reflecting} (for $\lsp'(c)$).  We also prove an
analogous theorem for the union of various spectra of this kind
(theorem~\ref{thm:lsp-rig-combo}). We note that the dimension is
irrelevant for the length spectral rigidity results; if the sound
speed is fixed, the length spectrum is independent of dimension.

Using proposition~\ref{prop:rce}, we find the following corollaries:

\begin{corollary}
Let $B=\bar B(0,1)\setminus \bar B(0,R)\subset\R^n$, $n\geq2$ and
$R\geq0$, be an annulus (or a ball if $R=0$).  Fix $\eps>0$ and let
$g_\tau$, $\tau\in(-\eps,\eps)$, be a $C^{1,1}$-regular non-trapping
$SO(n)$-invariant Riemannian metric making the boundary strictly
convex and satisfying the countable conjugacy condition and depending
$C^1$-smoothly on the parameter $\tau$.  If $\lsp(g_\tau)=\lsp(g_0)$
for all $\tau\in(-\eps,\eps)$, then there are rotation equivariant
diffeomorphisms $\psi_\tau\colon B\to B$ so that $\psi^*_\tau
g_\tau=g_0$ for all $\tau\in(-\eps,\eps)$.
\end{corollary}

In dimension $n \geq 3$ all rotation equivariant diffeomorphisms
(diffeomorphisms commuting with the action of $SO(n)$) are radial. In
dimension $n=2$ the diffeomorphisms are also radial if the metrics are
assumed to be $O(2)$-invariant.

\begin{corollary}
Let $B=\bar B(0,1)\setminus \bar B(0,R)\subset\R^3$, and $R\geq0$, be
an annulus (or a ball if $R=0$).  Fix $\eps>0$ and let $g_\tau$,
$\tau\in(-\eps,\eps)$, be a $C^{\infty}$-regular non-trapping rotation
invariant Riemannian metric making the boundary strictly convex and
satisfying the countable conjugacy condition and depending
$C^1$-smoothly on the parameter $\tau$.
Assume also that the length spectrum is non-degenerate up to rotations.
If $\spec(g_\tau)=\spec(g_0)$
for all $\tau\in(-\eps,\eps)$, then there are radial diffeomorphisms
$\psi_\tau\colon B\to B$ so that $\psi^*_\tau g_\tau=g_0$ for all
$\tau\in(-\eps,\eps)$.
\end{corollary}

\begin{remark}
While our main results are stated for the Laplace--Beltrami operator,
they are equally valid for the spectra associated to the toroidal
modes (see~\cite[chapter 8.6]{DahlTromp} for a precise definition) of an elastic operator. In section~\ref{sec: efunction and
  quant}, we point out that the length spectrum may be recovered from
either the spectrum of the Laplace--Beltrami operator or the toroidal
mode eigenfrequencies so the results above hold when one has the
spectrum associated to toroidal modes. In fact, our trace formula
(proposition~\ref{prop:Trace Formula}) to recover the length spectrum
from the spectrum is done for the more general toroidal modes and
frequencies. Nevertheless, we show how our proof and result holds for
the Laplace--Beltrami operator as well. We note that the spectrum of
the Laplace--Beltrami operator depends on dimension.
\end{remark}

The proofs of these theorems require three key ingredients which we
elaborate in the next subsection:
\begin{enumerate}
\item[$\bullet$] A trace formula relating the Neumann spectrum of the
  Laplacian to the length spectrum. Thus, we will only need to prove
  length spectral rigidity since spectral rigidity (theorem~\ref{thm:
    intro spec rigidity}) follows from the trace formula.
\item[$\bullet$] Sufficiently many periodic broken rays are stable
  under geometry-preserving perturbations of the metric and the
  derivative of the length of such broken rays is the periodic broken
  ray transform of the variation of the metric.
\item[$\bullet$] The periodic broken ray transform uniquely determines
  a radial function. 
\end{enumerate}

\subsection*{Outline of the proof}

The breakup of the paper is as follows. In section~\ref{sec: efunction
  and quant}, we discuss the relevant partial differential operators
and their eigenfunctions. We also discuss geodesics in spherical
symmetry and state the trace formula that we will prove (see
proposition~\ref{prop:Trace Formula}). Section~\ref{sec: trace
  formula} will be devoted to a proof of the trace formula. It will be
necessary to use several transforms and the Debye expansion (see
Appendix~\ref{sec:Debye exp}) to convert the Green's function for the
wave propagator written in terms of eigenfunctions to the dynamical
Fourier integral operator (FIO) representation analogous to the one
in \cite{DG75} and~\cite{Mel79}. Along the way, the connection from
eigenfunction to geodesics becomes rather explicit. To reinforce this
point, in section~\ref{sec: Melrose/guill construction} we revisit the
wave propagator constructed in~\cite{Mel79} and show how all our
explicit calculations relate to the abstract, geometric construction
in that paper. Finally, a standard application of the method of
steepest descent and stationary phase provides the leading order
asymptotics for the trace.

After proving the trace formula, we prove the rigidity of the length
spectrum in section~\ref{sec: spectral rigidity proof}.  Together with
the trace formula this implies the rigidity of the Neumann spectrum of the
Laplace--Beltrami operator.  We have a family of radially symmetric
wave speeds $c_\tau$ parameterized by $\tau\in(-\eps,\eps)$.  For any
periodic broken ray that is locally stable under the family of
deformations, the derivative of its length is the integral of the
metric variation over the periodic broken ray.  In the case of closed
manifolds and periodic geodesics this is well known, and in the case
of non-periodic broken rays this was observed
in~\cite{IS:brt-pde-1obst}.  Since the length spectrum is independent
of the parameter $\tau$, these derivatives vanish.  The countable
conjugacy condition (definition~\ref{def:ccc}) guarantees that
sufficiently many periodic broken rays are stable, so that we may
conclude that the periodic broken ray transform of the function
$\left.\Der{\tau}c_\tau^{-2}\right|_{\tau=0}$ vanishes.  It then
follows from recent results of periodic broken ray tomography on
spherically symmetric manifolds~\cite{dHI:radial-xrt} that the
function in question has to vanish.  Since the function is radial,
this can be seen as an injectivity result for an Abel-type integral
transform.  Consequently $c_\tau$ is independent of $\tau$ and the
rigidity of the length spectrum and thus the spectrum follows.

\subsection*{Acknowledgements}

M.V.d.H.\ gratefully acknowledges support from the Simons Foundation
under the MATH + X program and the National Science Foundation under
grant DMS-1559587. J.I.\ was supported by the Academy of Finland
(decision 295853), and he is grateful for hospitality and
support offered by Rice University during visits. V.K.\ thanks the
Simons Foundation for financial support.
We would also like to thank Gunther Uhlmann for helpful discussions and providing us useful references.

\section{Geodesics and Eigenfunctions} \label{sec: efunction and
quant}

In this section, we describe the relevant partial differential
operators, associated eigenfunctions, and the connection between
toroidal modes and the spectrum of the Laplace--Beltrami operator. We
state our trace formula (proposition~\ref{prop:Trace Formula}) along
with an important remark related to the Laplace--Beltrami operator
described in the introduction. First, let us provide a preliminary
discussion of geodesics in spherically symmetric manifolds.

\subsection{Geodesics in a spherically symmetric model}

For the moment, we suppose that $n=2$ and equip the annulus $M=\bar
B(0,1)\setminus\bar B(0,R)$ with spherical coordinates. For a maximal
geodesic we define its radius as its Euclidean distance to the origin.
We let $\gamma(r)$ be the maximal geodesic of radius $r$ which has its
tip (the closest point to the origin) at the angular position
$\theta=0$.

For $r_0\in(R,1)$, the geodesic $\gamma(r_0)$ can be parametrized as
\begin{equation}
[-L(r_0),L(r_0)]\ni t\mapsto (r(t),\theta(t))
\end{equation}
so that $r(0)=r_0$ and $\theta(0)=0$.
Here $L(r_0)>0$ is the half length of the geodesic.
Using the conserved quantities $c(r(t))^{-2}[r'(t)^2+r(t)^2\theta'(t)^2]=1$
(squared speed) and $p = p_{\gamma} \coloneqq c(r(t))^{-2}r(t)^2\theta'(t)=r_0/c(r_0)$ (angular momentum)
one can find the functions $r(t)$ and $\theta(t)$ explicitly.

Using these conserved quantities it is straightforward to show that
\begin{equation}
\label{eq:L}
L(r)
=
\int_r^1\frac{1}{c(r')}\left(1-\left(\frac{rc(r')}{r'c(r)}\right)^2\right)^{-1/2}\der
r' = \int_r^1\frac{1}{c(r')^2\wkb(r';p)} \der r',
\end{equation}
where $\wkb(r;p) \coloneqq \sqrt{c^{-2}(r) - r^{-2}p^2}$. We introduce the quantity $\wkb$ because it will appear naturally in the asymptotic approximations of eigenfunctions in section~\ref{sec: trace formula} and its relation to geodesics will now be clear.

We denote $\alpha(r)\coloneqq\theta(L(r))$, where $\theta$ is the angular
coordinate (taking values in~$\R$) of the geodesic $\gamma(r)$.
That is, $2\alpha(r)$ is the angular distance of the endpoints of $\gamma(r)$.
It may happen that $\alpha(r)>\pi$ if the geodesic winds around the origin several
times.
Using the invariants given above, one can also find an explicit formula for
$\alpha(r)$:
\begin{equation}
\label{eq:alpha}
\alpha(r)
=
\int_r^1\frac{rc(r')}{c(r)(r')^2}\left(1-\left(\frac{rc(r')}{r'c(r)}\right)^2\right)^{-1/2}\der r'
=
\int_r^1\frac{p}{(r')^2\wkb(r';p)} \der r'
.
\end{equation}

We will use the following lemma without mention whenever we need
regularity of these functions.

\begin{lemma} \label{lma:alpha-L-regularity}
When $c$ is $C^{1,1}$ and satisfies the Herglotz condition, then the
functions $\alpha$ and $L$ are $C^{1}$ on $(R,1)$.
\end{lemma}

\begin{proof}
This follows from equations (35), (67), and (68) and proposition~15
in~\cite{dHI:radial-xrt}.
\end{proof}

\subsection{Toroidal modes, eigenfrequencies, and traces}

We now use spherical coordinates $(r,\theta,\psi)$. Toroidal modes are
precisely the eigenfunctions of the isotropic elastic operator that
are sensitive to only the shear wave speed. We forgo writing down the
full elastic equation, and merely write down these special
eigenfunctions connected to the shear wave speed (full details with
the elastic operator may be found in~\cite[chapter 8.6]{DahlTromp}). Analytically, these eigenfunctions admit a separation in
radial functions and real-valued spherical harmonics, that is,
\begin{equation}
   u = {}_n\mathbf{D}_l Y^m_l ,
\end{equation}
where
\begin{equation}
   \mathbf{D} = U(r)\ (- k^{-1})
      [-\widehat{\theta} (\sin \theta)^{-1} \partial_{\psi}
                + \widehat{\psi} \partial_{\theta}] ,
\end{equation}
in which $k = \sqrt{l (l + 1)}$ (instead of the asymptotic Jeans
relation, $k = l + \tfrac{1}{2}$) and $U$ represents a radial function
(${}_n U_l$). In the further analysis, we ignore the curl (which
signifies a polarization); that is, we think of ${}_n\mathbf{D}_l$ as
the multiplication with ${}_n U_l(-k^{-1})$. In the above, $Y^m_l$ are
spherical harmonics, defined by
\[
   Y^m_l(\theta,\psi) = \left\{ \ba{rcl}
   \sqrt{2} X^{\abs{m}}_l(\theta) \cos(m \psi) & \text{if}
                             & -l \le m < 0 ,
\\
   X^0_l(\theta) & \text{if} & m = 0 ,
\\
   \sqrt{2} X^m_l(\theta) \sin(m \psi) & \text{if} & 0 <  m \le l ,
   \ea\right.
\]
where
\[
   X^m_l(\theta) = (-)^m \sqrt{\frac{2l + 1}{4\pi}}
               \sqrt{\frac{(l-m)!}{(l+m)!}} P^m_l(\cos\theta) ,
\]
in which
\[
   P^m_l(\cos(\theta)) = (-)^m \frac{1}{2^l l!} (\sin\theta)^m
     \left( \frac{1}{\sin\theta} \frac{\mathrm{d}}{\mathrm{d}\theta}
     \right)^{l+m} (\sin\theta)^{2l} .
\]
The function, $U$ (a component of displacement), satisfies the
equation
\begin{equation}\label{eq: equation for U_2}
   [-r^{-2} \partial_r\ r^2 \mu \partial_r
       + r^{-2} \partial_r\ \mu r - r^{-1} \mu \partial_r
         + r^{-2} (-1 + k^2) \mu] \, U - \omega^2 \rho U = 0 ,
\end{equation}
where $\mu = \mu(r)$ is a Lam\'{e} parameter and $\rho = \rho(r)$ is
the density, both of which are smooth. Also, $\omega = {}_n\omega_l$
denotes the associated eigenvalue. Here, $l$ is referred to as the
angular order and $m$ as the azimuthal order.

The traction is given by
\begin{equation}\label{eq: Neumann condition for U_2}
   T(U) = \mathcal{N} U ,\qquad
   \mathcal{N} = \mu \p_r - r^{-1}\mu
\end{equation}
which vanishes at the boundaries (Neumann condition). The radial
equations do not depend on $m$ and, hence, every eigenfrequency is
degenerate with an associated $(2l + 1)$-dimensional eigenspace
spanned by
\[
   \{ Y^{-l}_l,\ldots,Y^l_l \} .
\]

We use spherical coordinates $(r_0,\theta_0,\psi_0)$ for the location,
$x_0$, of a source, and introduce the shorthand notation
$({}_n\mathbf{D}_l)_0$ for the operator expressed in coordinates
$(r_0,\theta_0,\psi_0)$. We now write the (toroidal contributions to
the) fundamental solution as a normal mode summation
\begin{equation}\label{normal-mode-summation}
   G(x,x_0,t) = \operatorname{Re}\
          \sum_{l=0}^{\infty} \sum_{n=0}^{\infty}\
          {}_n\mathbf{D}_l ({}_n\mathbf{D}_l)_0\
   \sum_{m=-l}^l Y^m_l(\theta,\psi) Y^m_l(\theta_0,\psi_0)\
        \frac{e^{\ii {}_n\omega_l t}}{\ii {}_n\omega_l} .
\end{equation}
On the diagonal, $(r,\theta,\psi) = (r_0,\theta_0,\psi_0)$ and, hence,
$\Theta = 0$.
Here $\Theta$ is the angular epicentral distance, cf.~\eqref{eq: eq for big Theta}.
We observe the following reductions in the evaluation of
the trace of~\eqref{normal-mode-summation}:
\begin{itemize}
\item
The functions $U(r)$ are normalized, so that
\begin{equation} \label{eq:U2norm}
   \int_{\CMB}^{\surf}
        U(r)^2 \rho(r) r^2 \dd r = 1 .
\end{equation}
Meanwhile, the spherical harmonic terms satisfy
\begin{equation} \label{eq:Ylmnorm}
   \sum_{m=-l}^l \iint
      Y^m_l(\theta,\psi)^2 \sin \theta \dd\theta \dd\psi
           = 2l + 1
\end{equation}
(counting the degeneracies of eigenfrequencies).
\item
If we were to include the curl in our analyis (generating vector
spherical harmonics), taking the trace of the matrix on the diagonal
yields
\begin{equation} \label{eq:GYlmnorm}
   \sum_{m=-l}^l \iint
   (-k^{-2})
      \abs{[-\widehat{\theta} (\sin \theta)^{-1} \partial_{\psi}
                + \widehat{\psi} \partial_{\theta}]
       Y^m_l(\theta,\psi)}^2 \sin \theta \dd\theta \dd\psi
           = 2l + 1 .
\end{equation}
\end{itemize}

\noindent
From the reductions above, we obtain
\begin{equation}
   \int_M
      G(x,x,t) \, \rho(x) \dd x
      = \sum_{l=0}^{\infty} \sum_{n=0}^{\infty}\
        (2l + 1)
        \operatorname{Re} \left\{
        \frac{e^{\ii {}_n\omega_l t}}{
                               \ii {}_n\omega_l} \right\}
\end{equation}
or
\begin{equation} \label{eq:TrptG}
  \Tr(\p_t G)(t) = \int_M
      \partial_t G(x,x,t) \, \rho(x) \dd x
      = \sum_{l=0}^{\infty} \sum_{n=0}^{\infty}\
        (2l + 1)
        \operatorname{Re} \left\{
        e^{\ii {}_n\omega_l t} \right\} .
\end{equation}
We now write
\begin{equation*}
   {}_n f_l(t) = \operatorname{Re} \left\{
       \frac{e^{\ii {}_n\omega_l t}}{
                               \ii {}_n\omega_l} \right\}
\end{equation*}
which is the inverse Fourier transform\WriteupNote{There was a missing
  factor of $\frac12$ in ${}_n\hat{f}_l(\omega)$.} of
\begin{equation}
   {}_n \hat{f}_l(\omega) = \frac{1}{2\ii {}_n\omega_l}
      \left[ \mstrut{0.5cm} \right.
      \pi \delta(\omega - {}_n\omega_l)
           - \pi \delta(\omega + {}_n\omega_l)
                              \left. \mstrut{0.5cm} \right] .
\end{equation}
Moreover, taking the Laplace--Fourier transform yields
\begin{equation} \label{eq:nflom}
   \int_0^{\infty} {}_n f_l(t) e^{-\ii \omega t} \, \dd t
   = \frac{1}{2\ii {}_n\omega_l}
    \left[ \mstrut{0.5cm} \right.
    \frac{\ii}{-(\omega - {}_n\omega_l) + \ii 0}
    - \frac{\ii}{-(\omega + {}_n\omega_l) + \ii 0}
            \left. \mstrut{0.5cm} \right] .
\end{equation}
This confirms that the trace is equal to the inverse Fourier transform
of
\[
   \sum_{l=0}^{\infty} (2l + 1) \sum_{n=0}^{\infty}
   \frac{1}{2\ii {}_n\omega_l}
      \left[ \mstrut{0.5cm} \right.
      \pi \delta(\omega - {}_n\omega_l)
           - \pi \delta(\omega + {}_n\omega_l)
                              \left. \mstrut{0.5cm} \right] .
\]

In the next subsection, we explain how the toroidal eigenfrequencies
$\{{}_n\omega_l\}_{n,l}$ relate to the Neumann spectrum of the
Laplace--Beltrami operator described in the introduction. We also show
why all our results and proofs in connection to the trace formula
(proposition~\ref{prop:Trace Formula}) hold for this spectrum as well.

\subsection{Connection between toroidal eigenfrequencies, spectrum
        of the Laplace--Beltrami operator, and the Schr\"{o}dinger
        equation} \label{sec: connect to LB}

We relate the spectrum of a scalar Laplacian, the eigenvalues
associated to the vector valued toroidal modes, and the trace
distribution $\sum_{l=0}^{\infty} \sum_{n=0}^{\infty}\ (2l+1)
\cos(t{}_n\omega_l)$.

We note that \eqref{eq: equation for U_2} and \eqref{eq: Neumann
  condition for U_2} for $U$ ensure that $v = U Y^m_l$ satisfies
\begin{equation} \label{eq: scalar P}
   P v \coloneqq \rho^{-1}
       (-\nabla \cdot \mu \nabla + P_0) v = \omega^2 v ,\qquad
   \mathcal{N} v = 0 \text{ on }\p M
\end{equation}
where $P_0 = r^{-1}(\p_r\mu)$ is a $0$th order operator and $\omega^2$
is a particular eigenvalue. Hence $U Y^m_l$ are scalar eigenfunctions
for the self-adjoint (with respect to the measure $\rho\dd x$) scalar
operator $P$ with Neumann boundary conditions (on both boundaries)
expressed in terms of $\mathcal{N}$.

The above argument shows that we may view the toroidal spectrum
$\{{}_n\omega^2_l\}_{n,l}$ as also the collection of eigenvalues
$\lambda$ for the boundary problem on scalar functions (\ref{eq:
  scalar P}). Thus (\ref{eq:TrptG}) can be written in the form
\begin{equation}\label{eq: toroidal trace equal laplace trace}
   \Tr \, (\p_t G)
          = \sum_{\lambda \in \spec(P)} \cos(t \sqrt{\lambda}) ,
\end{equation}
where the last sum is taken with multiplicities for the
eigenvalues. (While $G$ is a vector valued distribution, the
asymptotic trace formula we obtain is for $\Tr (\p_t G)$, which is equal
to $\sum_{\lambda \in \spec(P)} \cos(t \sqrt{\lambda})$ by the
normalizations we have chosen.) Up to principal symbols, $P$ coincides
with the $\Delta_c = c^3 \nabla \cdot c^{-1} \nabla$ upon identifying
$c^2$ with $\rho^{-1} \mu$. This means that the length spectra of $P$
and $\Delta_c$ will be the same even though they have differing subprincipal symbols and spectra. Thus, the trace formula which will
appear to have a unified form, connects two different spectra to a
common length spectrum and the proof is identical for both.

For concreteness, we recall~\cite[theorem 1]{Mel79}. Suppose that
$\lambda_1 \leq \lambda_2 \leq \dots \to \infty$ denote the Neumann
spectrum of the Laplace--Beltrami operator $\Delta_g$. We form the
distribution
\begin{equation}\label{eq: Melrose cosine trace}
   \sum_{\lambda \in \spec(\Delta_g)} \cos(t \sqrt{\lambda}) .
\end{equation}
Under certain geometric conditions (where there is no symmetry) for a
simpler manifold described there, the authors prove the following: Let
$T$ be the singular support of \eqref{eq: Melrose cosine trace} and
the only closed geodesics, $\gamma$, of period $T$ satisfy certain
geometric conditions and have Maslov indices $\sigma_\gamma$. Then for
$t$ near $T$, \eqref{eq: Melrose cosine trace} is equal to the real
part of
\begin{equation}\label{eq: Melrose trace}
   \sum_{T_\gamma = T} \ii^{\sigma_\gamma}
       \frac{T^\sharp_\gamma}{(|I-P_\gamma|)^{1/2}}
            (t-T+0 \ii)^{-1} + L^1_{loc} ,
\end{equation}
where $T^\sharp$ is the primitive period and $P_\gamma$ is a certain
Poincar\'{e} map described there. Our results will settle the question
of whether there is a formula analogous to \eqref{eq: Melrose trace}
(same as~\cite[(1.3)]{Mel79}) for the distributions $\sum_{\lambda \in
  \spec(P)} \cos(t \sqrt{\lambda})$ and $\sum_{\lambda \in
  \spec(c)} \cos(t \sqrt{\lambda})$ in our spherically
symmetric manifold with boundary, encompassing a ball and an
annulus\footnote{The ball is representative of Earth's inner core
  while the annulus is representative for Earth's mantle.}.

We will prove a trace formula using a WKB expansion of
eigenfunctions. To this end, it is convenient to establish a
connection with the Schr\"{o}dinger equation. Indeed, we present an
asymptotic transformation finding this connection. In boundary normal
coordinates $(r,\theta)$ (which are spherical coordinates in dimension
three by treating $\theta$ as coordinates on the $2$-sphere),
\begin{equation}
   P = \rho^{-1} (-r^{-2} \p_r r^2 \mu \p_r
                  - \mu r^{-2} \Delta_\theta + P_0) ,
\end{equation}
where $\Delta_\theta$ is the Laplacian on the $2$-sphere. We write
$\lambda = \omega^2$ as before and rewrite the second-order equation \eqref{eq: scalar P}
in the form
\begin{align}
Y&= \left(\ba{l}
       r v \\ r \mathcal{N} v
       \ea\right) ,
\label{eq:Ydef}\\
\p_r Y &= A(r,\omega) Y ,
\label{eq:fors}
\intertext{where}
   A_{11} &= 2 r^{-1} ,
\nonumber\\
   A_{21} &= (-\Delta_\theta - 2) r^{-2} \mu - \omega^2 \rho ,
\nonumber\\
   A_{12} &= \mu^{-1} ,
\nonumber\\
   A_{22} &= -2 r^{-1} ,
\nonumber
\end{align}
satisfying
\[
   J A = -A^* J\quad\text{with}\quad
   J = \left(\ba{rr} 0 & 1 \\ -1 & 0 \ea\right) ,\quad J^2 = -I .
\qquad
\]
The Neumann condition is applied at $r = \surf$ and at $r = \CMB$. In
preparation of the asymptotic analysis we will instead invoke the
transformation using the same notation (by abuse of notation,
cf.~\eqref{eq:Ydef})
\begin{equation}
   Y = \left(\ba{l}
       r v \\ \omega^{-1} r \mathcal{N} v
       \ea\right) .
\end{equation}
With this definition, the matrix $A$ in~\eqref{eq:fors} admits the
expansion
\begin{equation}
   A(r,\omega) = \omega \, [A_0(r) + \omega^{-1} A_1(r)
                                         + \omega^{-2} A_2(r)] ,
\end{equation}
\begin{equation*}
   A_0(r) = \begin{pmatrix} 0 & \mu^{-1} \\
            -\rho - r^{-2} \omega^{-2}\mu\Delta_\theta
                              & 0 \end{pmatrix} ,
\quad
   A_1(r) = \begin{pmatrix} 2 r^{-1} & 0 \\
                                       0 & -2 r^{-1} \end{pmatrix} ,
\end{equation*}
\begin{equation*}
   A_2(r) = \begin{pmatrix} 0 & 0 \\ -2 r^{-2} \mu & 0 \end{pmatrix} ,
\end{equation*}
viewing $\omega$ as a large parameter. A similarity transform gives
\begin{equation}
   A_0 R = R \Lambda ,\qquad
   \Lambda(r) = \begin{pmatrix} 0 & 1 \\
                                -\tilde{\wkb}^2 & 0 \end{pmatrix} ,\ \
   R(r) = \begin{pmatrix} \mu^{-1/2} & 0 \\
                          0 & \mu^{1/2} \end{pmatrix} ,
\end{equation}
where we have defined
\begin{equation}
   \tilde\wkb^2 = \rho(r) \mu(r)^{-1}
                 + \omega^{-2}r^{-2} \Delta_\theta .
\end{equation}
We then seek an asymptotic transformation $L$ and $W$
\[
   Y = L W\quad\text{with}\quad
   L(r,\omega) = \sum_{j \ge 0} L_j(r) \omega^{-j} ,
\]
so that the original matrix system implies
\begin{equation}
   \partial_r W = \omega \Lambda(r) W .
\end{equation}
Substitution and equating terms with equal powers of $\omega^{-1}$
gives
\begin{eqnarray}
   A_0 L_0 - L_0 \Lambda &=& 0 ,
\\
   A_0 L_1 - L_1 \Lambda &=& \partial_r L_0 - A_1 L_0 ,
\end{eqnarray}
and so on. We find a simple solution
\begin{equation}
   L_0 = r^{-1} R(r).
\end{equation}
Then
\begin{equation}
\begin{split}
L_0 W &= \left(\ba{l}
       \mu^{-1/2} w \\ \omega^{-1} \mu^{1/2} \partial_r w
       \ea\right) ,\\
   r v &= \mu^{-1/2} w ,\\
   r \mathcal{N} v &= \mu^{1/2} \partial_r w ,\\
   W_1 &= r w ,\\
   W_2 &= \omega^{-1} r \partial_r w .
\end{split}
\end{equation}
Here, $w$ satisfies the equation
\begin{equation}
   (\partial_r^2  + r^{-2} \Delta_\theta
         + \omega^2 \rho \mu^{-1}) w = 0 ,\qquad
   \p_r w = 0\ \text{ on }\ \p M .
\end{equation}
If $y(\theta)$ is an eigenfunction of $\Delta_\theta$ with eigenvalue
$-k^2$ and $V = V(r)$ is radial function, we choose $w = w_k = V(r)
y(\theta)$ so that $V(r)$ must now satisfy
\begin{equation}\label{eq:VSchro}
   \p_r^2 V + \omega^2 \wkb^2 V = 0 ,\quad
   \p_r V = 0\ \text{ on }\ \p M ,
\end{equation}
where $\wkb = \rho(r) \mu(r)^{-1} - \omega^{-2}r^{-2}k^2$, generating
two linearly independent solutions. The WKB asymptotic solution to
this PDE with Neumann boundary conditions will precisely give us the
leading order asymptotics for the trace formula, and is all that is
needed. We note that
\[
   r U = \mu^{-1/2} V\ \text{ and }\ r T(r) = \mu^{1/2} \p_r V .
\]
For the boundary condition, we note that we would end up with the same
partial differential equation with different boundary conditions for
$V$ in the previous section if we had used the boundary condition
$\p_r u = 0 \text{ on }\p M$. Indeed, one would merely choose
$\mathcal{N}u = \mu \p_r u$ instead without the $0$'th order
term. However, the boundary condition for $V$ would be of the form
\[
   \p_r V = K(r) V\quad\ \text{ on }\ \p M
\]
with $K$ signifying a smooth radial function. Nevertheless, the
leading order (in $\omega$) asymptotic behavior for $V$ stays the same
despite the $K$ term as clearly seen in the calculation of section
\ref{sec: WKB eigenfunctions}. Thus, our analysis applies with no
change using the standard Neumann boundary conditions. This should
come as no surprise since in~\cite{Mel79}, the $0$'th order term in the
Neumann condition played no role in the leading asymptotic analysis of
their trace formula. Only if one desires the lower-order terms in the
trace formula would it play a role.

\subsection{Poincar\'{e} maps and the trace formula}
\label{sec:Poincare and trace}

Here, we describe the relevant Poincar\'{e} map that will appear in
the trace formula and state the trace formula that we will prove. Let
$ \mathbb{R} \ni t \to \gamma(t)$ be a periodic broken
bicharacteristic in $S^*M$ of period $T>0$ (see~\cite{Mel79} for the
relevant definitions). It is associated to the metric
$c^{-2}(\abs{x})e$, where $c$ is a smooth radial function, $e$ is the Euclidean metric,
and $\gamma$ undergoes reflections in $\p S^* M$ according to Snell's
law. We also denote by $\Phi^T \colon S^* M \to S^* M$ the broken
bicharacteristic flow of $T$ units of time as described in
\cite{Mel79}. Its fixed point set is given by
\[
   C_T \coloneqq \{ \eta \in S^* M ; \Phi^T(\eta) = \eta \}
\]
and without loss of generality, we assume that $C_T$ is connected, for
otherwise, we would look at a connected component instead. We impose
the \emph{clean intersection hypothesis} appearing in
\cite{DG75,Mel79} so that $C_T$ is a submanifold for any $T$ and $T_mC_T = \ker(\Id-
\der\Phi^T(m))$ at each point $m \in C_T$.
This holds for all periodic orbits if and only if $c$ satisfies the periodic conjugacy condition (definition~\ref{def:pcc}) which requires that the endpoints of a single maximal geodesic segment of a periodic broken ray are never conjugate; see remark~\ref{rmk:pcc}.

By construction, the image of $\gamma$ belongs to $C_T$. There is an
obvious symplectic group action of $SO(3)$ on $T^*M$ under which the
Hamiltonian $\frac{1}{2} c^2 \abs{\xi}^2$ is invariant; here, $\xi$
denotes the dual variable to $x$. Thus, for each $g \in SO(3)$, the
set $g \cdot \Im(\gamma)$ given by the group action also belongs
to $C_T$. Assuming that $c$ has no other symmetries (follows from the Herglotz condition) and the periodic conjugacy condition, all elements of
$C_T$ are obtained this way. This is because a periodic orbit will
either fail to be periodic or not have period $T$ after a small
perturbation of the angular momentum $p$; hence, $p$ remains constant
on $C_T$. Thus, $C_T$ may be parameterized by $\mathbb{R}_t \times
SO(3)$, revealing that under the Herglotz and periodic conjugacy conditions
\[
   \dim(C_T) = 4 .
\]
The Herglotz condition ensures that the group action never coincides
with the geodesic flow; without it, the dimension could quite possibly
be smaller. Hence, $T_m S^* M / T_m C_T$ is only one-dimensional and
we obtain an induced map on the quotient space,
\[
   I - \der\Phi^T \colon T_m S^* M / T_m C_T \to T_m S^* M / T_m C_T .
\]
We denote the equivalence class of all closed orbits of period $T$
related to $\gamma$ by an element of $SO(3)$ or by a time reversal of
$\gamma$ by $[\gamma]$. We write the above map as $I - P_{[\gamma]}$
and refer to $P_{[\gamma]}$ as the \emph{Poincar\'{e} map} associated
to the equivalence class of $\gamma$. Now, $I - P_{[\gamma]}$ will end
up being an isomorphism, whose determinant at each point $m \in C_T$
will stay invariant. Hence, the quantity $\abs{I - P_{[\gamma]}}^{-1}
\coloneqq \abs{\det(I -P_{[\gamma]})}^{-1}$ is well defined as a
single number associated to $[\gamma]$.

In the above, $\gamma$ may be multiple revolutions of another closed
orbit of minimal period called the \emph{primitive orbit} associated
to $\gamma$, which has a primitive period denoted $T^\sharp$. Note
that $T$ will merely be a positive integer multiple of $T^\sharp$. In
spherical symmetry, $\gamma$ is confined to a disk and it must be a
concatenation of geodesic segments that travel from the outer boundary
$r = \surf$ to either the reflection point $r = \CMB$ or the turning
point $r = \rturn$ (see section~\ref{sec: spectral rigidity proof} for
details). We let $N_\gamma$ denote the number of these segments
comprising the primitive orbit associate to $\gamma$.

We now state our proposition pertaining to the trace formula.

\begin{proposition}\label{prop:Trace Formula}
Suppose the radial wave speed $c$ satisfies the Herglotz condition and the periodic conjugacy condition (definition~\ref{def:pcc}).
Assume additionally that the length spectrum is non-degenerate in the sense that any two periodic broken rays of the same length are rotations of each other.

The distribution $(\operatorname{Tr} \, (\p_t G))(t) =
\sum_{n,l}(2l+1)\cos(t{}_n\omega_l)$ is singular precisely at the
length spectrum. Suppose $T \in \operatorname{singsupp}(\Tr
(\p_t G))$ and let $d$ be the dimension of the fixed point set for
$\Phi^T$. Suppose that $\gamma$ is one of the broken periodic orbits
of period $T$. Then $d=4$ and for $t$ near $T$, the contribution of
$[\gamma]$ to the leading singularity of $(\Tr(\p_t
G))(t)$ is the real part of
\[
   (t-T+ \ii 0)^{-(d+1)/2} \left(\frac{1}{2\pi \ii}\right)^{(d-1)/2}
   \ii^{\sigma_{\gamma}} \abs{I-P_{[\gamma]}}^{-1/2}
        \frac{T^\sharp_{\gamma}}{2\pi N_{\gamma}} c_d \abs{SO(3)} ,
\]
where
\begin{enumerate}
\item[$\bullet$] $\sigma_{\gamma}$ is the KMAH index associated to
  $\gamma$ defined in~\cite{Mel79};
\item[$\bullet$] $c_d$ is a constant depending only on $d$;
\item[$\bullet$] $\abs{SO(3)}$ is the volume of the compact Lie group
  $SO(3)$ under the Haar measure.
\end{enumerate}
\end{proposition}

In appendix~\ref{sec:Trace with symmetry} we identify this trace
formula in the framework of manifolds with symmetries given by a
compact Lie group.
In appendix~\ref{app:edge-cases} we discuss some edge cases that justify our geometric assumptions.

\begin{remark}\label{rem: trace holds for laplace beltrami}
This trace formula is in fact a more general statement than that for
the Neumann Laplace--Beltrami operator, which is just a special case of
the above proposition.
\end{remark}

\begin{remark}
\label{rmk:tr-pcc-nondeg}
The periodic conjugacy condition and non-degeneracy of the length spectrum are needed to prove that the singular support of the trace is precisely the length spectrum.
However, they are not necessary for spectral rigidity or proving that the spectrum determines the length spectrum.

For unique determination of the length spectrum, it is enough that the primitive length spectrum (excluding all but primitive orbits) is non-degenerate.
Given the singularity at $T^\sharp$, we know what the singularities at $2T^\sharp,3T^\sharp,\dots$ will be.
If they are not as expected, then another broken ray must contribute a singularity at the same place, and we have found the next primitive length.
This allows to recover the primitive length spectrum and therefore the whole length spectrum from the (shapes and locations of) singularities in the trace.
But if two primitive lengths coincide, there is no way to distinguish the corresponding singularities.

If we drop the periodic conjugacy condition, then some periodic orbits may fail the clean intersection hypothesis.
This allows us to recover only a part of the length spectrum from the singularities, but this part is enough.
Such problematic periodic broken rays are ignored anyway in the proof of length spectral rigidity, since they might not be stable under deformations.
\end{remark}

\section{Proof of the Trace Formula (proposition~\ref{prop:Trace Formula})}
\label{sec: trace formula}

In this section, we prove the trace formula in the form of proposition
\ref{prop:Trace Formula} for the annulus. The idea behind the proof is
to construct rather explicitly a fundamental solution. First, we do
some preliminary analysis to manipulate $G$ into the right form before
taking its trace. Concretely, in subsection~\ref{sec: eigenfunction
  asymtotics}, we construct WKB eigenfunction solutions to get
explicit formulas for the leading order asymptotics of the
eigenfunctions. Afterwards, in subsection~\ref{sec:Poisson formula} we
use the classical Poisson summation formula and the Debye expansion to
write the leading order asymptotics for $G$ as a certain
propagator, which relates the eigenfunctions to geodesics in the
annulus. At that point, we use section~\ref{sec: Melrose/guill
  construction} to show how all our constructions are quite natural
and directly relate to the wave propagator appearing in
\cite{Mel79}. Finally, we complete the proof in section~\ref{sec:final
  proof of trace} by taking traces and carrying out the method of
steepest descent and stationary phase to obtain the desired asymptotic
formula appearing in proposition~\ref{prop:Trace Formula}.

In the further analysis, we employ the summation formula,
\begin{equation} \label{eq:SYtP}
   \sum_{m=-l}^l Y^m_l(\theta,\psi) Y^m_l(\theta_0,\psi_0)
           = \frac{(2l + 1)}{4\pi} \, P_l(\cos \Theta) ,
\end{equation}
where the $P_l$ are the Legendre polynomials, with $P_l(1) = 1$, and
$\Theta$ signifies the angular epicentral distance,
\begin{equation}\label{eq: eq for big Theta}
   \cos \Theta\ =\ \cos\theta \, \cos\theta_0
       + \sin\theta \, \sin\theta_0 \, \cos(\psi - \psi_0) .
\end{equation}


\begin{remark}
We note that $\hat{G}$ is the kernel of the resolvent in the
time-harmonic formulation. The normal mode summation becomes
\begin{equation} \label{eq:nmsom}
   \hat{G}(x,x_0,\omega) = \frac{1}{2\pi}\
          \sum_{l=0}^{\infty} \sum_{n=0}^{\infty}\
     (l + \tfrac{1}{2})\ {}_n \hat{f}_l(\omega)\ \underbrace{
        {}_n\mathbf{D}_l ({}_n\mathbf{D}_l)_0}_{\eqqcolon {}_n H_l}\
                   P_l(\cos \Theta) ,
\end{equation}
explicitly showing the eigenfrequencies as simple poles
(cf.~\eqref{eq:nflom}).
\end{remark}

We abuse notation and denote
\[
   {}_n H_l = k^{-2} U(r) U(r_0)
\]
in the formula for $G$ to not treat the curl operations at first. This
will not cause a risk of confusion since we will specify the exact
moment that we apply the curl operators, which will be just before
taking the trace in subsection~\ref{sec:final proof of trace}.

\subsection{Asymptotic analysis of
         eigenfunctions} \label{sec: eigenfunction asymtotics}

We describe the radial eigenfunctions and their asymptotic expansions
for general $k \in \RR$. We then introduce the dispersion relations,
$\omega_n(k)$.

\subsubsection{WKB eigenfunctions}\label{sec: WKB eigenfunctions}

Here, we consider asymptotic solutions, $V$, to \eqref{eq:VSchro}.
Depending on $p$, we distinguish the following regimes:
\begin{itemize}
\item
\textbf{Evanescent} ($\surf / c(\surf) < p < \infty$). Here,
$\wkb^2(r) < 0$, and the solution is always non-oscillatory, that is,
evanescent. We do not obtain eigenfunctions.
\item
\textbf{Diving} ($\CMB/c(\CMB) < p < \surf/c(\surf))$: We summarize
the WKB solution of~\eqref{eq:VSchro} in the vicinity of a general
turning point. A turning point, $r = \rturn$, is determined by
\[
   \wkb^2(\rturn) = 0 .
\]
Near a turning point, $r \approx \rturn$, and
\[
   \wkb^2(r) \simeq \alpha_0 (r - \rturn) .
\]
Away from a turning point,
\[
   \wkb^2 > 0\text{ if $r \gg \rturn$} ,\quad
   \wkb^2 < 0\text{ if $r \ll \rturn$} .
\]
Matching asymptotic solutions yields
\begin{equation}
B
\begin{cases}
\abs{\wkb}^{-1/2} \exp\left(-\omega
\int_r^{\rturn} \abs{\wkb} \, \dd r\right)
,
& r \ll \rturn
\\[.5em]
2 \pi^{1/2} \alpha_0^{-1/6} \omega^{1/6}
\mathrm{Ai}(- \omega^{2/3} \alpha_0^{1/3} (r - \rturn)),
& r \simeq \rturn
\\[.5em]
2 \wkb^{-1/2} \cos \left(-\omega \int_{\rturn}^r \wkb \, \dd r
   - \pi/4 \right)
,
& r \gg \rturn
.
\end{cases}
\end{equation}
From these one can obtain a uniform expansion, that is, the Langer
approximation
\begin{equation}
\begin{split}
   V(r,\omega;p) &= 2 \pi^{1/2} \chi^{1/6} (-\wkb^2)^{-1/4}
                   \mathrm{Ai}(\chi^{2/3}(r)) ,\\
   \chi(r) &= -(3/2) \omega \int_{\rturn}^r (-\wkb^2)^{1/2} \dd r ,
\end{split}
\end{equation}
valid for $r \in [\CMB,\surf]$. One obtains eigenfunctions
corresponding with turning rays.

Up to leading order, where
$r \gg \rturn$,
\begin{align}
   V &= 2 B \wkb^{-1/2}
       \cos \left(\omega \int_{\rturn}^r \wkb \, \dd r'
                              - \pi/4 \right),
\\
   \partial_r V &= -2 \omega B \wkb^{-1/2}
       \sin \left(\omega \int_{\rturn}^r \wkb \, \dd r'
                              - \pi/4 \right).
\end{align}
Here, $B$ is obtained from the normalization~\eqref{eq:U2norm}, which
requires the uniformly asymptotic solution over the entire interval
$[\CMB, \surf]$. Applying the
Riemann--Lebesgue lemma, one obtains
\begin{equation} \label{eq:RLII}
   \int_\CMB^\surf
        U(r)^2 \rho(r) r^2 \dd r \simeq 2 B^2
     \int_{\rturn}^{\surf} \wkb^{-1} c^{-2} \dd r .
\end{equation}
Here, $\int_{\rturn}^{\surf} \wkb^{-1} c^{-2} \dd r$ can be
identified with the half one-return travel time, $\tfrac{1}{2} T$, say.
\item
\textbf{Reflecting} ($0 < p < \CMB / c(\CMB))$: The solutions are
oscillatory in the entire interval $[\CMB, \surf]$ ($\wkb^2(r;p) >
0$), correspond with reflecting rays, and are of the form
\begin{align}
   V &= C \wkb^{-1/2} \exp\left( \ii \omega
       \int_{\CMB}^r \wkb \dd r' \right)
       + D \wkb^{-1/2} \exp\left( -\ii \omega
       \int_{\CMB}^r \wkb \dd r' \right),
\\[0.25em]
   \partial_r V &= \ii \omega C \wkb^{1/2}
                  \exp\left( \ii \omega
       \int_{\CMB}^r \wkb \dd r' \right)
       - \ii \omega D \wkb^{1/2} \exp\left( -\ii \omega
       \int_{\CMB}^r \wkb \dd r' \right).
\end{align}
to leading order. Imposing the Neumann condition,
$T(\CMB) = 0$, implies that $D = C$. The
constant $C$ is obtained from the normalization~\eqref{eq:U2norm}
using the oscillatory solution over the entire interval
$[\CMB, \surf]$. Applying the
Riemann--Lebesgue lemma yields
\begin{equation} \label{eq:RLIII}
   \int_{\CMB}^{\surf}
        U(r)^2 \rho(r) r^2 \dd r \simeq 2 C^2
     \int_{\CMB}^{\surf}
            \wkb^{-1} c^{-2} \dd r .
\end{equation}
Here, $\int_{\CMB}^{\surf} \wkb^{-1} c^{-2} \dd r$ can be identified
with the half one-return reflection travel time, $\tfrac{1}{2} T$,
say.
\end{itemize}

\subsubsection{Boundary condition and dispersion relations}

We backsubstitute $p = \omega^{-1} k$ in $\wkb$. Imposing the
Neumann boundary condition yields: 
\begin{itemize}
\item
\textbf{Diving} ($\CMB /
c(\CMB) < p < \surf /
c(\surf))$:
\begin{equation} \label{eq:efII}
   \omega \int_{\rturn}^{\surf}
            \wkb(r';\omega^{-1} k) \, \dd r'
                      = \left( n + \frac{5}{4} \right) \pi
\quad
   \text{($n$ is the overtone index)} .
\end{equation}
\item
\textbf{Reflecting} ($0 < p < \CMB /
c(\CMB)$):
\begin{equation} \label{eq:efIII}
   \omega \int_{\CMB}^{\surf}
            \wkb(r';\omega^{-1} k) \dd r'\ =\ (n+1) \pi
\quad
   \text{($n$ is the overtone index)} .
\end{equation}
\end{itemize}
All these (radial) quantization-type conditions yield solutions ${}_n
\omega_k \eqqcolon \omega_n(k)$. Using the implicit function theorem, we
introduce $k_n = k_n(\omega)$ as the solution of
\[
   \omega - \omega_n(k_n) = 0 .
\]

We revisit the general relation between phase and group velocities. We
have
\[
   c_n(k) = \frac{\omega_n(k)}{k}
\]
and
\[
   C_n = \pdpdo{(c_n k)}{k} = c_n + k \, \pdpdo{c_n}{k} .
\]
The corresponding ray parameter is given by
\begin{equation} \label{eq:pcnk}
   p = \frac{k}{\omega_n(k)} = \frac{1}{c_n(k)} .
\end{equation}

\subsection{Poisson's summation formula}\label{sec:Poisson formula}

\subsubsection{Analytic continuation}

The Legendre equation is of second order and hence admits two linearly
independent solutions: Legendre functions of the first and second
kind, $P_{\lambda}$, $Q_{\lambda}$, with integral representations
\begin{align}
   P_{\lambda}(z) &= \frac{2}{\pi} \operatorname{Im}
     \int_0^{\infty} \left[ z -
        \ii \sqrt{1 - z^2} \cosh(t) \right]^{-\lambda - 1} \, \dd t ,
\label{Pint}\\
   Q_{\lambda}(z) &= \phantom{\frac{2}{\pi}} \operatorname{Re}
     \int_0^{\infty} \left[ z -
        \ii \sqrt{1 - z^2} \cosh(t) \right]^{-\lambda - 1} \, \dd t ,
\label{Qint}
\end{align}
for $z \in [-1,1]$. These integral representations are valid for
$\operatorname{\lambda} > -1$. Analytic continuation of the integral
representation for $P_{\lambda}$, $Q_{\lambda}$ into the region
$\operatorname{Re} \{ \lambda \} \le -1$ follows the relations
\begin{eqnarray}
   P_{-\lambda-1} &=& P_{\lambda} ,
\label{Pc}\\
   Q_{-\lambda-1} &=&
           Q_{\lambda} - \pi \, \cot(\lambda \pi) \, P_{\lambda} .
\label{Qc}
\end{eqnarray}
The bottom equation implies that $Q_{\lambda}$ has simple poles at the
negative integers $\lambda = -1, -2, \ldots$. The Legendre function
of the first kind, $P_{\lambda}$, coincides with the Legendre
polynomials for $\lambda = l = 0, 1, 2, \ldots$ (cf.~\eqref{eq:SYtP}).

\subsubsection{Application of Poisson's formula}

Poisson's formula is given by
\begin{equation} \label{eq:P}
   \sum_{l=0}^{\infty} f(l + \tfrac{1}{2})
        = \sum_{s = -\infty}^{\infty}  (-)^s \int_0^{\infty} f(k)
                 e^{-2 \ii s k \pi} \, \dd k .
\end{equation}


\begin{remark}
Poisson's formula can be obtained from the Watson transformation: If
$f$ is a function analytic in the vicinity of the real axis, and
$\mathcal{C} = \mathcal{C}^- \cup \mathcal{C}^+$ is a contour around
the positive real axis, then
\begin{equation} \label{eq:W}
   \sum_{l=0}^{\infty} f(l + \tfrac{1}{2})
        \ =\ \tfrac{1}{2} \int_{\mathcal{C}} f(k) [\cos(k \pi)]^{-1}
                 e^{-\ii k \pi} \, \dd k .
\end{equation}
The integrand in the right-hand side has simple poles at $(2n+1) / 2$,
$n \in \N_0$ -- where $\cos(k \pi) = 0$. It follows from the residual
theorem. Poisson's formula is obtained as follows. In the limit $s \to
\infty$ the path of integration is $\mathcal{C}^-$, while in the limit
$s \to -\infty$ the path of integration is $\mathcal{C}^+$. One then
expands
\[
   [\cos(k \pi)]^{-1}
         = \frac{2 e^{-\ii k \pi}}{1 + e^{-2 \ii k \pi}}
\]
in a series separately for $\operatorname{Im} \{ k \} > 0$ and $<
0$.
\end{remark}


We apply Poisson's formula to the summation in $l$ in~\eqref{eq:nmsom}
while keeping the summation in $n$ intact. We obtain
\begin{equation} \label{eq:PSS}
   \hat{G}(x,x_0,\omega) = \frac{1}{2\pi}\
          \sum_{s=-\infty}^{\infty}\ (-)^s
     \int_0^{\infty} \left[
     \sum_{n=0}^{\infty} \hat{f}_n(k;\omega)\ H_n(k) \right]
          P_{k - 1/2}(\cos \Theta) e^{-2 \ii s k \pi} \, k \dd k .
\end{equation}
This expression can be rewritten as
\begin{multline} \label{eq:PSS-poss}
   \hat{G}(x,x_0,\omega) = \frac{1}{2\pi}\
          \sum_{s=1}^{\infty}\ (-)^s
     \int_0^{\infty} \left[
     \sum_{n=0}^{\infty} \hat{f}_n(k;\omega)\ H_n(k) \right]
          P_{k - 1/2}(\cos \Theta)
        \{ e^{-2 \ii s k \pi} + e^{2 \ii s k \pi} \} \, k \dd k
\\
   + \frac{1}{2\pi}\ \int_0^{\infty} \left[
     \sum_{n=0}^{\infty} \hat{f}_n(k;\omega)\ H_n(k) \right]
          P_{k - 1/2}(\cos \Theta) \, k \dd k .
\end{multline}

\subsection*{Traveling-wave Legendre functions}

Traveling-wave Legendre functions are given by
\begin{align}
   Q_{\lambda}^{(1)} &= \frac{1}{2} \left( P_{\lambda}
                       + 2 \frac{\ii}{\pi} Q_{\lambda} \right) ,
\\
   Q_{\lambda}^{(2)} &= \frac{1}{2} \left( P_{\lambda}
                       - 2 \frac{\ii}{\pi} Q_{\lambda} \right) .
\end{align}
Using~\eqref{Pc}--\eqref{Qc} yields the continuation relations
\begin{align}
   Q_{-\lambda-1}^{(1)} &= Q_{\lambda}^{(1)}
                 - \ii \, \cot(\lambda \pi) \, P_{\lambda} ,
\label{Q1c}\\
   Q_{-\lambda-1}^{(2)} &= Q_{\lambda}^{(2)}
                 + \ii \, \cot(\lambda \pi) \, P_{\lambda} .
\label{Q2c}
\end{align}
Both have simple poles at the negative integers.


\begin{remark}
Their asymptotic behaviors, as $\abs{\lambda} \gg 1$, are (assuming that
$\Theta$ is sufficiently far away from the endpoints of $[0,\pi]$)
\begin{align}
   Q_{\lambda-1/2}^{(1)} &\simeq
     \left( \frac{1}{2\pi \lambda \sin \Theta} \right)^{1/2}
              e^{-\ii (\lambda \Theta - \pi/4)} ,
\label{Q1tr}\\
   Q_{\lambda-1/2}^{(2)} &\simeq
     \left( \frac{1}{2\pi \lambda \sin \Theta} \right)^{1/2}
              e^{+\ii (\lambda \Theta - \pi/4)} ,
\label{Q2tr}
\end{align}
upon substituting $z = \cos \Theta$. Taking into consideration the
time-harmonic factor $e^{\ii \omega t}$, it follows that $Q^{(1)}$
represents waves traveling in the direction of increasing $\Theta$,
while $Q^{(2)}$ represents waves traveling in the direction of
decreasing $\Theta$.
\end{remark}


To distinguish the angular directions of propagation, one decomposes
\begin{equation} \label{PP}
   P_{k-1/2}(\cos \Theta) = Q_{k-1/2}^{(1)}(\cos \Theta)
                  + Q_{k-1/2}^{(2)}(\cos \Theta).
\end{equation}
Substituting~\eqref{PP} into~\eqref{eq:PSS}, we get
\begin{multline}
   \hat{G}(x,x_0,\omega) = \frac{1}{2\pi}\
          \left[ \mstrut{0.6cm} \right.
   \sum_{s = 1,3,5,\ldots} (-)^{(s-1)/2}
\\
   \int_0^{\infty} \left[
     \sum_{n=0}^{\infty} \hat{f}_n(k;\omega)\ H_n(k) \right]
          Q_{k - 1/2}^{(1)}(\cos \Theta)
          \{ e^{-\ii (s-1) k \pi} - e^{\ii (s+1) k \pi} \}
                   \, k \dd k
\\
   + \sum_{s = 2,4,\ldots} (-)^{s/2}
     \int_0^{\infty} \left[
     \sum_{n=0}^{\infty} \hat{f}_n(k;\omega)\ H_n(k) \right]
          Q_{k - 1/2}^{(2)}(\cos \Theta)
          \{ e^{-\ii s k \pi} - e^{\ii (s-2) k \pi} \}
                   \, k \dd k  \left. \mstrut{0.6cm} \right] .
\end{multline}
In the application of~\eqref{eq:P}--\eqref{eq:W}, the contour
$\mathcal{C}^-$ is followed for the series containing $e^{-\ii (s-1) k
  \pi}$ and $e^{-\ii s k \pi}$, while the contour $\mathcal{C}^+$ is
followed for the series containing $e^{\ii (s+1) k \pi}\ \ $ and
$e^{\ii (s-2) k \pi}$.

\subsubsection{Preparation for the method of steepest descent}

In preparation of the application of the method of steepest descent,
we rewrite the $k$-integral from $\RR_{\ge 0}$ to $\RR$. We make use
of the analytic extension of $Q_{k - 1/2}^{(1,2)}$ from real positive
to real negative $k$ values. With~\eqref{Q1c}--\eqref{Q2c} we find that
\begin{equation}
   Q_{-k - 1/2}^{(1,2)}(\cos \Theta)
     = e^{\pm 2 \ii k \pi} Q_{k - 1/2}^{(1,2)}(\cos \Theta)
       + e^{\pm \ii k \pi} \tan(k \pi) P_{k - 1/2}(-\cos \Theta) .
\end{equation}
Using the symmetry of $\hat{f}_n(.;\omega)$, $H_n(.)$, it follows that
the integrals over $P_{k - 1/2}(\cos \Theta)$ cancel. Then
\begin{multline} \label{eq:GintQk}
   \hat{G}(x,x_0,\omega) = \frac{1}{2\pi}\
   \sum_{s = 1,3,5,\ldots} (-)^{(s-1)/2}
\\
   \int_{-\infty}^{\infty}
   \left[ \mstrut{0.6cm} \right. \sum_{n=0}^{\infty}
          \hat{f}_n(k;\omega)\ H_n(k) \left. \mstrut{0.6cm} \right]
          Q_{k - 1/2}^{(1)}(\cos \Theta)
          e^{-\ii (s-1) k \pi} \, k \dd k
\\
   + \frac{1}{2\pi}\ \sum_{s = 2,4,\ldots} (-)^{s/2}
   \int_{-\infty}^{\infty}
   \left[ \mstrut{0.6cm} \right. \sum_{n=0}^{\infty}
          \hat{f}_n(k;\omega)\ H_n(k) \left. \mstrut{0.6cm} \right]
          Q_{k - 1/2}^{(2)}(\cos \Theta)
          e^{-\ii s k \pi} \, k \dd k .
\end{multline}
The integrands in the terms of these series can be identified as wave
constituents travelling along the surface or boundary, the
representations of which can be obtained by techniques from
semi-classical analysis. Indeed, $s$ is referred to as the
(multi-orbit) arrival number while we distinguish the orientation of
propagation in the two series. The term $s = 1$ corresponds with waves
that propagate from source to receiver along the minor arc; the term
$s = 2$ corresponds with waves that propagate from source to receiver
along the major arc. At $\Theta = 0$ and $\Theta = \pi$ the traveling
wave Legendre functions have logarithmic singularities, namely $\log
\Theta$ at $\Theta = 0$ and $\log (\pi - \Theta)$ at $\Theta = \pi$;
the singularities cancel when taking the sums together.

\subsubsection{Traveling wave expansion}

Here, we apply to~\eqref{eq:GintQk} the Debye expansion described in
appendix~\ref{sec:Debye exp} to obtain a form more closely resembling
a wave propagator:
\begin{multline}
   \sum_{n=0}^\infty \hat{f}_n(k;\omega) H_n(k)
\\
   = k^{-2} (rr_0c(r)c_0(r))^{-1}(\rho(r)\rho(r_0)
         \wkb(r;p)\wkb(r_0;p))^{-1/2}
     \frac{1-c_n^{-1}C_n}{2\ii \omega}
\\
   \sum_{i=1}^\infty \exp\left[-\ii \omega \tau_i(r,r_0;p)
            + \ii N_i \frac{\pi}{2}\right] ,
\end{multline}
where the $\tau_i$ are time-delay functions which we will relate to
the travel time of a geodesic, and $N_i$ are integers contributing to
the KMAH index as described in appendix~\ref{sec:Debye exp}.

Next, we change variables of integration from $k$ to $p$. We
encounter the Jacobian (cf.~\eqref{eq:pcnk})
\[
   \pdpdo{p}{k} = \frac{1}{c_n}
           \left( \frac{C_n}{c_n} - 1 \right) \frac{1}{k} ,
\]
so that
\[
   p^{-1} \dd p = (1 - c_n^{-1} C_n) \, k^{-1} \dd k .
\]
The path of integration is beneath the real axis, while taking $\omega
> 0$. After making the substitution into~\eqref{eq:GintQk}, we obtain
\begin{align*}
   &\hat{G}(x,x_0,\omega)
\\
   &\quad = \frac{1}{4\pi}\
   \sum_{s = 1,3,5,\ldots} (-)^{(s-1)/2}
   (rr_0c(r)c_0(r))^{-1}(\rho(r)\rho(r_0))^{-1/2}
\\
   &\quad\quad \int_{-\infty}^{\infty}
   (\wkb(r;p)\wkb(r_0;p))^{-1/2}
   \left[ \mstrut{0.6cm} \right. \sum_{i=0}^{\infty}
          \exp\left[-\ii \omega \tau_i(r,r_0;p)
   + \ii N_i \frac{\pi}{2}\right] \left. \mstrut{0.6cm} \right]
\\
   &\quad\quad\quad\quad
          Q_{\omega p - 1/2}^{(1)}(\cos \Theta)
          e^{-\ii \omega(s-1) p \pi} \, p^{-1} \dd p
\\
   &\quad  + \frac{1}{4\pi}\ \sum_{s = 2,4,\ldots}
  (-)^{s/2}(rr_0c(r)c_0(r))^{-1}(\rho(r)\rho(r_0))^{-1/2}
\\
   &\quad\quad \int_{-\infty}^{\infty}
   (\wkb(r;p)\wkb(r_0;p))^{-1/2}
   \left[ \mstrut{0.6cm} \right. \sum_{i=0}^{\infty}
         \exp\left[-\ii \omega \tau_i(r,r_0;p)
   + \ii N_i \frac{\pi}{2}\right] \left. \mstrut{0.6cm} \right]
\\
   &\quad\quad\quad\quad
          Q_{\omega p - 1/2}^{(2)}(\cos \Theta)
          e^{-\ii \omega s p \pi} \, p^{-1} \dd p .
\end{align*}

\begin{remark}
To explicitly reveal the connection with geometrical optics we
consider individual terms in~\eqref{eq:GintQk}, upon changing the
variable of integration. We consider the first term, substitute the
resummation inside the integral, and insert the leading order
expansion,
\[
   Q_{\omega p - 1/2}^{(1)}(\cos \Theta) \simeq
     \left( \frac{1}{2\pi \omega p \sin \Theta} \right)^{1/2}
              e^{-\ii (\omega p \Theta - \pi/4)}
\]
(cf.~\eqref{Q1tr}) to obtain (cf.~\eqref{eq:GintQk})
\begin{multline}\label{eq: wave propagator form}
   \frac{1}{2\pi} (-)^{(s-1)/2}
     \int_{-\infty}^{\infty} \left[ \sum_{n=0}^{\infty}
     \frac{U_{n}(r;p) U_{n}(r_0;p)}{\omega_n^2 - \omega^2}
           \, (1 - c_n^{-1} C_n)^{-1} \right]
\\
     Q_{\omega p - 1/2}^{(1)}(\cos \Theta)
     e^{-\ii \omega (s-1) p \pi} \, p^{-1} \dd p
\\
   \simeq \frac{1}{4\pi} (-)^{(s-1)/2}
       (r r_0 c(r) c(r_0))^{-1}
       (2\pi \rho(r) \rho(r_0) \sin \Theta)^{-1/2}
\hspace*{3.0cm}
\\
   \int (\wkb(r;p) \wkb(r_0;p))^{-1/2}
   \sum_{i=1}^{\infty}
           \exp[-\ii \omega (\tau_i(r,r_0;p) + p \Theta + (s-1) p \pi)]
\\
   \exp[\ii (\pi/4) (2 N_i - 1)] (\omega p)^{-3/2} \dd p .
\end{multline}
\end{remark}

In the next section, we review the abstract argument appearing in
\cite{Mel79} in order to justify why we consider the above expression
a wave propagator. It will also motivate the method of steepest
descent and stationary phase calculation that we perform in section
\ref{sec:final proof of trace}.

\subsection{Melrose--Guillemin wave propagators}
\label{sec: Melrose/guill construction}

In~\cite{Mel79}, Guillemin and Melrose show how the Neumann half-wave
propagator, $e^{\pm \ii t\sqrt{-\Delta_g}}$, may be written as a sum
of operators denoted $\tilde{V}^i_\pm(t)$, such that for a fixed $t$,
it is a canonical graph FIO whose canonical relation is a certain
billiard map in phase space that we will briefly describe. In order to
avoid corners, they embed $M$ in a boundaryless manifold $\tilde{M}$
and then $\tilde{V}^i_+$ is an FIO on $\tilde{M} \times M$. It maps a
covector $(y,\eta) \in S^*M^o$ to the ``time $t$'' endpoint of a
broken bicharacteristic that undergoes $i$ reflections at points
$x_k(y, \eta) \in \p M$, $k=1,\dots,i$. Precisely, let $t_i =
t_i(y,\eta)$ denote the travel time along the broken bicharacteristic
to the last reflection point $(x_i,\xi_i) \in \p_+ S^*M$ and let
$(x_i,\xi^r_i)\in \p_-S^*M$ be the reflected covector pointing
``inside'' $M$. The canonical graph maps $(y,\eta)$ to
$\Phi^{t-t_i}(x_i,\xi_i^r)\in S^* \tilde{M}$ (here, $\Phi^t$ denotes
the bicharacteristic flow described in section~\ref{sec: efunction and
  quant}).  If fewer than $i$ reflections take place by time $t$ or it
maps outside of $M$, then $\tilde{V}^i_+(t)$ is microlocally smoothing
at such covectors. It will be convenient to denote a covector locally
in polar coordinates as $\eta = \abs{\eta}_{g(y)}\hat{\eta} \in
T^*_yM$. Then the corresponding conic Lagrangian may by parameterized
by a phase function of the form
\[
\phi_i(t,x,y,\eta) = \abs{\eta}_g(-t+S_i(x,y, \hat{\eta})),
\]
where, essentially, $S_i$ gives the travel time between points $x,y$
of the broken geodesic undergoing $i$ reflections that starts at
$(y,\hat{\eta})$ (when one finds $\hat{\eta} \in S^*_yM$ that
minimizes $S_i$).

For the following calculations, all that is necessary is that
$\tilde{V}_+^i(t)$ is a canonical graph FIO and that $\phi_i$ is a
phase function that locally parameterizes the conic Lagrangian
associated to the canonical graph. Thus, the Schwartz kernel of
$\tilde{V}_+^i(t)$ is indeed given locally (as described in
\cite{Mel79}) by
\[
(2\pi)^{-n}\int e^{\ii(-t\abs{\eta}_g+S_i(x,y,\eta))}a(t,x,y,\eta) \dd\eta,
\]
where $a$ is a classical symbol of order $m=0$.
If we introduced spherical coordinates in $\eta$ with radial variable $\omega \coloneqq \abs{\eta}_g$ and took the leading order (homogeneous of degree $m=0$) part of the classical symbol $a$,
then it becomes clear that~\eqref{eq: wave propagator form}
has the same form after a Fourier transform $\mathcal{F}_{\omega \to t}$ with phase function $- \omega(\tau_i(r,r_0;p) + p
\Theta + (s-1)p\pi)$ corresponding to $S_i= \omega S_i(x,y,\hat \eta)$ above. The order in $\omega$ does not match yet because
we have not yet applied the curl operators nor $\p_t$. The similarities become clearer as we proceed with the stationary phase calculations for both operators.

After taking the trace of the full propagator, the contribution by
$\tilde{V}_+^i$ is
\[
(2\pi)^{-n}\int e^{\ii(-t\abs{\eta}_g+S_i(x,x,\eta))}a(t,x,x,\eta) \dd x\dd\eta.
\]

We change variables into polar coordinates by first defining a map
$f\colon \RR^+_{\omega} \times S^*M \to T^*M$,
\[
f(\omega,x,p) = (x,\omega p).
\]
Then $f^\star\abs{dT^*M} = \omega^{n-1}d\omega\abs{dS^*M}$.  Without
loss of generality, since we only consider the leading order terms, we
may assume that $a$ is an $m$-homogeneous (in $\eta$) symbol. After a
change of variables into polar coordinates (keeping in mind that $a$
is supported in a local coordinate chart where $S^*M$ is trivialized),
the above integral is:
\begin{equation}\label{eq: double int after trace}
(2\pi)^{-n}\int_0^{\infty} e^{-\ii\omega t} \omega^{n+m-1} \dd\omega \int_{M\times
S^{n-1}} e^{\ii\omega S_i(x,x,p)} a(t,x,x,p) \dd x \dd p.
\end{equation}

\subsection*{Stationary phase with a degenerate phase function}

To obtain the leading order asymptotics, we apply the method of
stationary phase. We denote the critical set of $S_i$ (viewed as
function on $S^*M$) locally as $C_{S_i} = \{(x,p) \in S^*M; d_{x,p}S =
(d_xS_i(x,x,p)+d_yS_i(x,x,p), d_pS_i)= 0 \}.$ We may break up this set
into a countable number of disjoint connected components given by the
values of $S_i$:
\[
C_{S_i} = \bigcup _{k=1}^\infty C_{T_{ik}}, \
C_{T_{ik}} = \{ (x,p) \in S^*M; S_i=T_{ik}, \ d_{x,p}S_i = 0\}.
\]

By construction of the phase function, $C_{T_{ik}}$ is precisely the
fixed point set of the bicharacteristic flow $\Phi^{T_{ik}}$. Let
$d_{ik}$ be the codimension of $C_{T_{ik}}$. If $C_{T_{ik}}$
intersects $\p M$ transversely, then we may find independent
coordinates $z=z_1, \dots, z_{2n-1}$ such that $C_{T_{ik}}$ is given
by the vanishing of $z_1,\cdots, z_{d_{ik}}$ and $\p M$ is given by
$z_{2n-1} =0$.  For notation purposes, it will be convenient to denote
$z_I = (z_1,\cdots,z_{d_{ik}})$,
$z_{II}=(z_{d_{ik}+1},\cdots,z_{2n-2})$ and $\mathcal{O}$ a neighborhood where
such coordinates are valid.

We assume that $S_i$ is nondegenerate in the directions normal to the critical
 manifold $C_{T_{ik}}$, so that after a change of coordinates still denoted by the
 same letters, $S_i$ is locally given by
\[
S_i(z)= T_{ik}+\frac{1}{2}\sum_{r,s = 1}^{d_{ik}}\langle
\p_{z_r}\p_{z_s}S_i(0,z_{II},z_{2n-1})z_r, z_s \rangle.
\]
This follows precisely from the generalized Morse lemma. The Jacobian due to such a
change of coordinates is identically $1$ when restricted to $C_{T_{ik}}$ and so we
exclude it on our analysis to have simpler formulas.

To proceed, let $J = \abs{\frac{\p(x,p)}{\p z}}$ be the Jacobian factor
resulting from the change of variables and $\rho_\mathcal{O}$ a smooth localizer supported in
$\mathcal{O}$. Denoting $\tilde{a} = aJ$, the contribution of the inner integral of~\eqref{eq:
double int after trace} within $\mathcal O$ is
\begin{align*}
&(2\pi)^{-n}\int_0^{\infty}\int e^{\ii\omega S_i(z)}\rho_\mathcal{O} \tilde{a}(z)\dd z'\dd z_{2n-1}
\\
&\qquad= (2\pi)^{-n}\int_0^{\infty}\int e^{\ii\omega(T_{ik}+\frac{1}{2}\sum_{r,s =
1}^{d_{ik}}\langle \p_{z_r}\p_{z_s}S_i(0,z_{II},z_{2n-1})z_r, z_s
\rangle)}\rho_\mathcal{O}\tilde{a}(z)\dd z_I\dd z_{II}\dd z_{2n-1}
\\
&\qquad\sim \omega^{-d_{ik}/2}e^{\ii\omega T_{ik}}
(2\pi)^{(d_{ik}-2n)/2}\int_0^{\infty}\int
\tilde{a}^\star(0,z_{II},z_{2n-1})\dd z_{II}\dd z_{2n-1} + O(\omega^{-d_{ik}/2-1}).
\end{align*}

In the above, we applied stationary phase in the variables $z_I$ corresponding to normal
directions of $C_{T_{ik}}$. This requires the normal Hessian of $S_i$ which is
non-degenerate:
\[
d^2_{N}S_i \coloneqq (\p_{z_i}\p_{z_j}S)_{i,j=1,\dots,d_{ik}}|_{z_I=0},
\]
so that the leading amplitude after stationary phase becomes
\[
\tilde{a}^\star(t,0,z_{II},z_{2n-1}) = e^{\ii\frac{\pi}{4}\sgn( d^2_NS_i)}
\abs{d^2_NS_i}^{-1/2}\rho_\mathcal{O}\tilde{a}|_{C_{T_{ik}}} .
\]
One may also show that $J(0,z_{II},z_{2n-1})\dd z_{II}\dd z_{2n-1}$ is the
natural induced measure on $C_{T_{ik}}$ obtained from $dS^*M$ in local
coordinates. We label this measure as $\dd \mu_{ik}$.

At this point, we may actually write
\[
   a(t,z) = a(T_{ik},z) + O(t-T_{ik}),
\]
and use $-D_\omega e^{-i\omega(t-T_{ik})} =
(t-T_{ik})e^{-\ii\omega(t-T_{ik})}$.  Using integration by parts in
$\omega$, we may replace $a(t,z)$ by $a(T_{ik},z)$ since the
difference will lead to a term that is an order lower in $\omega$.
Substituting this into the inner integral above, setting
\[\alpha_{ik}=e^{\ii\frac{\pi}{4}\sgn( d^2_NS_i)}
\left(\frac{1}{2\pi}\right)^{(\dim(C_{T_{ik}})-1)/2}\int_{C_{T_{ik}}} a
\abs{d^2_NS_i}^{-1/2}  \ \dd \mu_{ik},\]
 the leading order term in our trace formula
becomes (with $m=0$ since $a$ is $0$th order)
\begin{equation*}
\begin{split}
&
\Re \sum_{ik} \alpha_{ik}\frac{1}{2\pi}\int_0^{\infty} e^{-\ii\omega(t-T_{ik})}
\omega^{n+m-1-d_{ik}/2} \dd \omega
\\&
=\Re \sum_{ik}\alpha_{ik}\frac{1}{2\pi}\int_0^{\infty} e^{-\ii\omega(t-T_{ik})}
\omega^{(\dim(C_{T_{ik}})-1)/2} \dd \omega .
\end{split}
\end{equation*}
 In the case of isolated periodic geodesics, one has $\dim(C_{T_{ik}})=1$. The above
 calculation would then be the real part of
\begin{align*}
\frac{1}{2\pi}\sum_{ik}\alpha_{ik}\int_0^{\infty} e^{-\ii\omega(t-T_{ik})} \dd\omega
&= \frac{1}{2\pi}\lim_{\epsilon \to 0}\sum_{ik}\alpha_{ik}
\frac{-1}{\ii(T_{ik}-t)-\epsilon}
=\frac{1}{2\pi} \sum_{ik}\alpha_{ik} \frac{\ii}{(T_{ik}-t)+\ii0}
\end{align*}
For the spherically symmetric case, $\dim(C_{T_{ik}}) = 4$ which
leads to a bigger singularity. In the next subsection, we apply
analogous computations for taking the trace in our setting.

\subsection{Proof of Proposition~\ref{prop:Trace Formula}}
\label{sec:final proof of trace}

We are finally ready to apply the method of steepest descent and
stationary phase to complete the proof of the trace formula.

\begin{proof}[Proof of proposition~\ref{prop:Trace Formula}]
We interchange the order of summation and integration, and invoke the
method of steepest descent. We carry out the analysis for a single
term, $s=1$. For $s=2,4,\ldots$ we have to add $s p \pi$ to $\tau_i$,
and for $s=3,5,\ldots$ we have to add $(s-1) p \pi$ to $\tau_i$,
in the analysis below.

We find (one or more) saddle points for each $i$, where
\[
   \partial_p \tau_i(r,r_0,p)|_{p = p_k} = -\Theta .
\]
Later, we will consider the diagonal, setting $r_0 = r$ and $\Theta =
0$. We label the saddle points by $k$ for each $i$ (and $s$). We note
that $r, r_0$ and $\Theta$ determine the possible values for $p$
(given $i$) which corresponds with the number of rays connecting the
receiver point with the source point (allowing conjugate points). For
$s=1$, the rays have not completed an orbit. With $s = 3$ we begin to
include multiple orbits.


%

We carry out a contour deformation over the saddles and obtain
\begin{equation} \label{eq:abws}
\begin{aligned}
   &\frac{1}{2\pi} (-)^{(s-1)/2}
     \int_{-\infty}^{\infty} \left[ \sum_{n=0}^{\infty}
     \frac{U_{2;n}(r;p) U_{2;n}(r_0;p)}{\omega_n^2 - \omega^2}
           \, (1 - c_n^{-1} C_n)^{-1} \right]
     Q_{\omega p - 1/2}^{(1)}(\cos \Theta)
\\
     &\hspace{1in}\qquad \qquad e^{-\ii \omega (s-1) p \pi} \, p^{-1} \dd p
\\
   &\hspace{1in}\simeq \frac{1}{4\pi} (-)^{(s-1)/2}
       (r r_0 c(r) c(r_0))^{-1}
       (2\pi \rho(r) \rho(r_0) \sin \Theta)^{-1/2}
\\[0.2cm]
   &\hspace{1in}\qquad \sum_i \sum_k
   \left[ \omega^{-2}p^{-3/2} (\wkb(r;.) \wkb(r_0;.))^{-1/2}
         \abs{\partial_p^2 \tau_i(r,r_0;.)}^{-1/2} \right]_{p = p_k}
\\
	&\hspace{1in}\qquad\qquad
                 e^{-\ii \omega T_{ik} + \ii M_{ik} (\pi/2)} ,
\end{aligned}
\end{equation}
where
\begin{equation}
\begin{split}
   T_{ik} &= T_{s;ik}(r,r_0,\Theta)
       = \tau_i(r,r_0;p_k) + p_k \Delta_s ,\\
   M_{ik} &= N_i - \tfrac{1}{2} (1 -
               \sgn \partial_p^2 \tau_i |_{p = p_k}) ,
\end{split}
\end{equation}
in which
\begin{equation}
   \Delta_s =
\begin{cases}
   \Theta + (s-1) \pi & \text{if $s$ is odd} \\
   -\Theta + s \pi    & \text{if $s$ is even.}
                   \end{cases}
\end{equation}
The $M_{ik}$ contribute to the KMAH indices, while the $T_{ik}$ represent
geodesic lengths or travel times. The orientation of the contour
(after deformation) in the neighborhood of $p_k$ is determined by
$\sgn \partial_p^2 \tau_i |_{p = p_k}$. We note that
\begin{itemize}
\item
$M_{ik} = M_{s;ik}(r,r_0,\Theta)$ for multi-orbit waves
  ($s=3,4,\ldots$) includes polar phase shifts produced by any angular
  passages through $\Theta = 0$ or $\Theta = \pi$ as well;
\item
if $r$ lies in a caustic, the asymptotic analysis needs to be adapted
in the usual way.
\end{itemize}


Next, we apply both curl operations $[-\hat \theta (\sin
  \theta)^{-1}\p_{\psi} + \hat \psi \p_{\theta}],[-\hat \theta_0 (\sin
  \theta_0)^{-1}\p_{\psi_0} + \hat \psi_0 \p_{\theta_0}]$ to each term
in the sum above and then tensor the vectors together in order to
obtain a sum of $2$-tensor fields. This will give us the actual normal
mode summation of \eqref{eq:nmsom}. Since we are interested in only
the leading order asymptotics in $\omega$, we need only consider these
operations to the term $\exp[-\ii\omega T_{ik}]$ which gives $(\omega
p_k)^2 \mathbf{D}(\theta,\psi,\theta_0,\psi_0)$ where $\mathbf{D}$ is
a 2-tensor field in the angular variables. We also apply an inverse
Fourier transform followed by $\p_t$ to the formula above (since we
are interested in the cosine propagator $\p_t G$) to obtain to leading
order
\begin{multline}
 \simeq \frac{1}{4\pi} (-)^{(s-1)/2}
       (r r_0 c(r) c(r_0))^{-1}
       (2\pi \rho(r) \rho(r_0) \sin \Theta)^{-1/2}
\\[0.2cm]
   \sum_i \sum_k
   \left[ p^{1/2} (\wkb(r;.) \wkb(r_0;.))^{-1/2}
         \abs{\partial_p^2 \tau_i(r,r_0;.)}^{-1/2} \right]_{p =
         p_k}\mathbf{D}(\theta,\psi,\theta_0,\psi_0)
\\
   \frac{1}{2\pi}\int_0^{\infty}\ii\omega  \exp[-\ii \omega( T_{ik}-t) +
               \ii M_{ik} (\pi/2)] \dd\omega,
\end{multline}
It will now be convenient to denote the above quantity by $F^{(1)}$ if
$s$ is even and $F^{(2)}$ if $s$ is odd.

Thus, we have a sum of kernels of FIOs associated with the wave
propagator. (Here, the summation over $i, k$ signifies the summation
over (broken) geodesics while the summation over $s$ signifies the
number of orbits, travelling clockwise or counterclockwise.)

Since we will need to restrict $\p_t G$ to the diagonal ($\Theta = 0$,
$r'=r$), we must be careful with the $(\sin \Theta)^{-1/2}$
term. Fortunately, this term comes from the asymptotics of $Q_{\omega
  p- 1/2}^{(1)}$ and $Q_{\omega p- 1/2}^{(2)}$ which merely represent
the two different directions of one particular geodesic. Of course, a
periodic broken orbit has the same period no matter which direction
one travels and in the trace formula, all orbits of a particular
period are combined. Hence, we can combine $Q_{\omega p- 1/2}^{(1)}$
and $Q_{\omega p- 1/2}^{(2)}$ near $\Theta=0$, which is a logarithmic
singularity that cancels when both of the functions are added
together, and this will not affect the trace formula.

Precisely, after substituting the asymptotic formula for
$Q^{(1,2)}_{\omega p_k - 1/2}(\cos \Theta)$, we may write for $s\geq
3$ odd
\begin{multline}
F^{(1,2)} \simeq \frac{1}{4\pi} \ii^{-1/2}(-)^{(s-1)/2}
       (r r_0 c(r) c(r_0))^{-1}
       (\rho(r) \rho(r_0))^{-1/2}(\sqrt{2\pi})
\\[0.2cm]
   \sum_i \sum_k
   \left[ p (\wkb(r;.) \wkb(r_0;.))^{-1/2}
         \abs{\partial_p^2 \tau_i(r,r_0;.)}^{-1/2} \right]_{p =
         p_k}\mathbf{D}(\theta,\psi,\theta_0,\psi_0)
\\
               \frac{1}{2\pi}\int_0^{\infty}\ii\omega^{3/2}  \exp[-\ii \omega(
               \tau_{ik}+p_k(s-1)\pi-t) + \ii M_{ik} (\pi/2)] Q^{(1,2)}_{\omega p_k
               - 1/2}(\cos(\Theta))\dd\omega,
\end{multline}
Asymptotically, as $\Theta \to 0$ one has
$Q^{(1)}_{\omega p_k - 1/2}(\cos(\Theta))+ Q^{(2)}_{\omega p_k -
1/2}(\cos(\Theta))
\simeq 1$ and so

 \begin{multline}
F^{(1)}+F^{(2)} \simeq -\frac{\pi}{(2\pi i)^{3/2}} (-)^{(s-1)/2}
       (r r_0 c(r) c(r_0))^{-1}
       (\rho(r) \rho(r_0))^{-1/2}
\\[0.2cm]
   \sum_i \sum_k
   \left[ p (\wkb(r;.) \wkb(r_0;.))^{-1/2}
         \abs{\partial_p^2 \tau_i(r,r_0;.)}^{-1/2} \right]_{p =
         p_k}\mathbf{D}(\theta,\psi,\theta_0,\psi_0)
\\
               \frac{1}{2\pi}\int_0^{\infty}i\omega^{3/2}  \exp[-\ii \omega(
               T_{ik}-t) + \ii M_{ik} (\pi/2)] \dd\omega,
\end{multline}
as $\Theta \to 0$.

It will be convenient to denote
\begin{multline}
A_{s;ik}(r,r_0,\Theta)
= (-)^{(s-1)/2}
       (r r_0 c(r) c(r_0))^{-1}
       (\rho(r) \rho(r_0))^{-1/2}
       \\
       \cdot \left[ p (\wkb(r;.) \wkb(r_0;.))^{-1/2}
         \abs{\partial_p^2 \tau_i(r,r_0;.)}^{-1/2} \right]_{p =
         p_k}D(\theta,\psi,\theta_0,\psi_0)
\end{multline}
where now $D$ is a scalar function, defined as the inner product of
the two vector fields that make up $\mathbf{D}$. One may check using
l'Hospital's rule and equation~\eqref{eq: eq for big Theta} that
$D|_{\theta = \theta_0,\psi = \psi_0} = 2$.

Next, we take the trace of $\p_t G$ by restricting to $(r=r_0,\Theta =
0)$ and integrating. The phase function on the diagonal is $T_{ik} =
\tau_i(r,r,p_k)+\pi(s-1)p_k$ and we apply stationary phase in the
variables $r,\theta,\psi$ with large parameter $\omega$. Since one has
$\p_p T_i(r,r,p) = 0$ at $p=p_k$, the critical points occur precisely
when
\[
\p_r T_i(r,r,p) + \p_{r_0}T_i(r,r,p) = 0, \qquad \p_p T_i(r,r,p) = 0.
\]
After a quick calculation, the first condition forces $T_i$ to be
independent of $r$. Also, we showed that for geodesics with turning
points, $U = O(\omega^{-\infty})$ when $r < \rturn$. Finally, using
the inverse Fourier transform,
\[
   \int_0^{\infty} \exp[i\omega(t-T)]\omega^{3/2} \dd\omega
        = c_d(t-T+\ii0)^{-5/2} ,\ \text{ with $c_d$ a constant.}
\]
Setting $R_{ik}= \{(r,\theta,\psi); r \in [\tilde{R},\surf], \;
\p_rT_i(r,r,p_k) + \p_{r_0}T_i(r,r,p_k)=0\}$ where $\tilde{R} =
\rturn$ or $\CMB$ depending on $p_k$, we have shown that in fact
$T_{ik}$ remains constant over $R_{ik}$ so that only certain $i$ are
allowable. We find that modulo terms of lower order in $\omega$, the
trace microlocally corresponds with~\cite{Langer}
\begin{equation}
   \operatorname{Re}\sum_s \sum_i \sum_k\
         \left(\frac{1}{2\pi \ii}\right)^{3/2}(t-T_{s;ik}+\ii0)^{-5/2}
       \ii^{M_{ik}}c_d\frac{1}{2} \int_{R_{ik}} A_{s;ik}(r,r,0) \rho r^2 \sin
       \theta \dd r
                 \dd\theta \dd\psi
\end{equation}
and we use \eqref{eq:RLII},~\eqref{eq:RLIII}. Here,
\[
   T_{s;ik} = T_{s;ik}(r,r;0) = \tau_i(r,r;p_k)
     + \left\{\begin{array}{rcl}
           p_k (s-1) \pi & \text{if $s$} & \text{odd} \\
           p_k s \pi     & \text{if $s$} & \text{even}
              \end{array}\right.
\]
is independent of $r$.
We note that $p_k$ exists only for $i$, and $s$, sufficiently large,
which reflects the geometrical quantization.

From this expression, it is clear that the singular support of the
trace consists of the travel times of periodic geodesics.

\begin{remark}
It is now apparent how the above formula relates to the trace formula
in~\cite{Mel79}. A term above for a certain $i,s,k$ index
corresponds to the trace of $\tilde{V}_i^\pm$ integrated over a
critical manifold $C_{T_{ik}}$ (which in our case is $R_{ik}$) as
described in section~\ref{sec: Melrose/guill construction}. In both
cases, the index $i$ is used to keep track of the number of
intersections of the broken ray with the boundary while the $k$ index
specifies a particular periodic ray and period.

The index $s$ in $T_{ik;s}$ describing travel time has no analog for
general manifolds and certainly does not appear in~\cite{DG75} and
\cite{Mel79}.  This is because the parameter $s$ is merely a byproduct
of using spherical coordinates to construct an FIO on the ball and
using the particular phase function we have. It is used to keep track
of the angular distance a particular geodesic has traversed in the
disk, since, while the angular distance is greater than $2\pi$ when a
particular geodesic traverses the full disk more than once, our
angular coordinates only range from $[0,2\pi]$. The factor
$(-)^{(s-1)/2}$ is part of the KMAH index corresponding to antipodal
conjugate points in the sphere that the associated geodesic (when
projected to $r=\surf$) passes through. The other part of the KMAH
index comes from $M_{ik}$.

The Poincar\'{e} map appearing in~\cite{DG75} and~\cite{Mel79}
corresponds to the factor $\abs{p_k^{-1}\p_p^2\tau_i|_{p=p_k}}^{-1/2}$
in our trace formula and the factor $\abs{I-P_{[\gamma]}}^{-1/2}$
described in section~\ref{sec: efunction and quant}. In~\cite{DG75},
this factor generally varies over the entire critical manifold
$C_{T_{ik}}$ when its dimension is greater than $1$, but will not in
our case due to the symmetry. In fact, this is how we know that this
factor corresponds to the Poincar\'{e} map: it must be a quantity that
remains constant over the critical manifold $R_{ik}$.
\end{remark}

We further simplify the above formula, that is, the integral involving
$A_{s;ik}$. First, since $T_{ik}$ is independent of $r$, then so is
$\tau_i(r,r;p) = \tau_i(p)$. Thus, we may pull $\p^2_p\tau_i$ out of
the integral involving $A_{s;ik}$ precisely because we are integrating
over a closed orbit:
\begin{align*}
 &\int_{R_{ik}} A_{s;ik}(r,r,0) \rho r^2 \sin \theta \dd r
                 \dd\theta \dd\psi
                 \\
 &\qquad \qquad =(-)^{(s-1)/2}\abs{p_k^{-1}\p^2_p\tau_i(p_k)}^{-1/2}\int_{R_{ik}}
 \frac{1}{c^2\beta(r,p_k)} 2\sin \theta
\dd r \dd\theta \dd\psi.
\end{align*}
We recall that the travel time $T$ for a piece of a geodesic from
$r_0$ ro $r$ is
\[
T = \int_{r_0}^r \frac{dr'}{c^2\beta(r',p_k)} \dd r'.
\]
Hence, denoting $T_{ik}^\sharp$ as the primitive period of the
geodesic, we obtain
\[
\int_{R_{ik}} \frac{1}{c^2\beta(r,p_k)} 2\sin \theta
\dd r \dd\theta \dd\psi = \frac{T_{ik}^\sharp}{N_{ik}} \int_0^{2\pi}\int_0^{\pi}
2\sin \theta
\dd\theta \dd\psi
= \frac{ T_{ik}^\sharp}{\pi N_{ik}}\abs{SO(3)},
\]
where $N_{ik}$ is the number of geodesic segments from $\tilde{R}$ to
$\surf$ along the primitive orbit with length $T^\sharp_{ik}$, and
$\abs{SO(3)}$ is the volume of $SO(3)$ under a Haar measure
\cite{Creagh92}.

Substituting these calculations, the leading order term in the trace
formula is
\begin{align} \label{eq: alpha_ik for spheric symm}
&\operatorname{Re}\sum_s \sum_i \sum_k\
         \left(\frac{1}{2\pi \ii}\right)^{3/2}(t-T_{s;ik}+\ii0)^{-5/2}
       \ii^{M_{ik}+s-1}c_d\\
&\qquad \qquad \qquad \qquad      \cdot \frac{T_{ik}^\sharp}{
N_{ik}}\abs{p_k^{-1}\p^2_p\tau_i(p_k)}^{-1/2}
\frac{1}{2\pi}\abs{SO(3)}. \qedhere
\end{align}
\end{proof}

\noindent
In the next section, we use the trace formula to prove our main
spectral rigidity theorems stated in the introduction.

\section{Proof of Spectral Rigidity}
\label{sec: spectral rigidity proof}

In this section we will prove that the length spectrum is rigid.  By
the trace formula this will imply that the spectrum of the
Laplace--Beltrami operator is rigid.

\subsection{Conjugacy conditions}

The following condition will be convenient:

\begin{definition}
\label{def:ccc}
We say that a $C^{1,1}$ sound speed $c$ satisfies the countable conjugacy condition
if there are only countably many radii $r\in(R,1)$ so that the endpoints of the
corresponding maximal geodesic $\gamma(r)$ are conjugate along that geodesic.
\end{definition}

Assuming this condition will eventually imply that the length spectrum is
countable.
Throughout this paper ``countable'' includes empty and finite sets, but for the
sake of brevity we shall not write ``at most countable''.

The next condition is directly related to the clean intersection property discussed earlier.
We will return to this condition in section~\ref{sec:spec-rig-proof}.

\begin{definition}
\label{def:pcc}
We say that the radial wave speed $c$ satisfies the \emph{periodic conjugacy condition} if $\alpha(r)\in\pi\Q$ implies $\alpha'(r)\neq0$.
Restating geometrically, this means that if a broken ray is periodic, then the endpoints of a geodesic segment are not conjugate along the segment.
\end{definition}

\begin{remark}
\label{rmk:pcc}
Consider a periodic broken ray of radius $r$.
It satisfies the clean intersection property mentioned in section~\ref{sec:Poincare and trace} if and only if either $\alpha'(r)\neq0$ (leading to $\dim(C_T)=4$) or $\alpha'$ vanishes in a neighborhood of $r$ (leading to $\dim(C_T)=5$).

If the second option is true, it will fail at each endpoint of the maximal interval on which $\alpha'$ vanishes.
Since $\lim_{r\to1}\alpha(r)=0$ when the boundary is strictly convex, such an endpoint exists.
Therefore, assuming the Herglotz condition, the clean intersection property for all periodic broken rays is equivalent with the periodic conjugacy condition of definition~\ref{def:pcc}.
We point out that the function $\alpha$ is not well defined and the dimension of the fixed point set can be $3$ if the Herglotz condition is violated.
\end{remark}

\subsection{Conditions for periodicity}

A geodesic can be extended into a broken ray.
The geodesic segments of a broken ray are rotations of each other.
It is easy to see that the broken ray corresponding to the geodesic $\gamma(r)$ is
periodic if and only if $\alpha(r)\in\pi\Q$.
We want to understand the set of radii $r\in(R,1)$ for which this is the case.

\begin{lemma}
\label{lma:ncp-implies-density}
Let $P\subset(R,1)$ be the set of radii for which the corresponding broken ray is
periodic.
Let $C\subset(R,1)$ be the set of radii $r$ for which the endpoints of the geodesic
$\gamma(r)$ are conjugate along $\gamma(r)$.
\begin{itemize}
	\item In fact $C=\{r\in(R,1);\alpha'(r)=0\}$.
	\item If $C$ has empty interior, then $P$ is dense.
\end{itemize}
\end{lemma}

\begin{proof}
The radius $r\in(R,1)$ parametrizes a family of geodesics.
By rescaling the speed (from the originally assumed unit speed), we may assume that
all geodesics are parametrized by $[-1,1]$.
Differentiating the geodesic $\gamma(r)$ with respect to $r$ gives a non-trivial
Jacobi field, a variation of a geodesic $\gamma(r)$.
The value of the Jacobi field at the endpoints of the geodesic describes the
movement of the endpoints of the geodesic under variation.
On the other hand, $\alpha(r)$ gives the endpoint of the geodesic.
Therefore the tangential component of the Jacobi field at the endpoint is
$\alpha'(r)$.
It follows from the reparametrization that the component normal to the boundary
vanishes.
Thus if $\alpha'(r)=0$, then the endpoints of $\gamma(r)$ are conjugate along
$\gamma(r)$.

On the other hand, if the endpoints are conjugate, then there is a
non-trivial Jacobi field vanishing at the endpoints. In dimension two
there can only be a one-dimensional space of such Jacobi fields, so
the Jacobi field in question must be symmetric in the time parameter
of the geodesic. Then we can identify the Jacobi field as
corresponding to a variation of the parameter $r$.  Combining this
with the previous observation shows that
\begin{equation}
r\in C
\iff
\alpha'(r)=0.
\end{equation}
This proves the first claim.

Let $R<a<b<1$. To prove the second claim, it suffices to produce
$r\in(a,b)$ so that $\alpha(r)\in\pi\Q$; see the discussion right
before the statement of this lemma.  For a contradiction, assume that
$\alpha(r)\notin\pi\Q$ for all $r\in(a,b)$. Since $\alpha$ is
continuous, this implies that $\alpha$ is constant on $(a,b)$. Thus
$\alpha'$ vanishes on $(a,b)$, so $C$ has interior -- a contradiction.
\end{proof}

\begin{lemma}
\label{lma:ccc}
If the sound speed $c\in C^{1,1}$ satisfies the countable conjugacy
condition (definition~\ref{def:ccc}) and the Herglotz condition, then
$C$ and $P$ are countable, $C$ is closed and $P$ is dense in $(R,1)$.
\end{lemma}

\begin{proof}
The countable conjugacy condition directly states that $C$ is countable.
Therefore $C$ cannot have interior, and density of $P$ follows from
lemma~\ref{lma:ncp-implies-density}.
Since $\alpha'$ is continuous, the preimage of zero under it, the set $C$, is
closed.

It remains to show that $P$ is countable. If it was uncountable, some
level set of $\alpha$ would be uncountable. An uncountable set has
uncountably many accumulation points and $\alpha'$ vanishes at every
accumulation point of a level set of $\alpha$.  This implies that $C$
is uncountable, which is impossible.
\end{proof}

\subsection{Length spectral rigidity}
\label{sec:lsp-rigidity}

The length spectrum of a manifold $M$ with boundary is the set of
lengths of all periodic broken rays on $M$.  If $M$ is a spherically
symmetric manifold as described above, we may choose whether or not we
include the rays that reflect on the inner boundary $r=R$.  If the
radial sound speed is $c$, we denote the length spectrum without these
interior reflections by $\lsp(c)$ and the one with these reflections
by $\lsp'(c)$. If the inner radius is zero ($R=0$), the manifold is
essentially a ball and the two kinds of length spectra coincide.

For clarity, we state the following three length spectral rigidity
theorems separately.

\begin{theorem}
\label{thm:lsp-rig-diving}
Let $B=\bar B(0,1)\setminus \bar B(0,R)\subset\R^n$, $n\geq2$ and
$R\geq0$, be an annulus (or a ball if $R=0$).  Fix $\eps>0$ and let
$c_\tau$, $\tau\in(-\eps,\eps)$, be a $C^{1,1}$ function
$[R,1]\to(0,\infty)$ satisfying the Herglotz condition and the
countable conjugacy condition and depending $C^1$-smoothly on the
parameter $\tau$.  If $R=0$, we assume $c_\tau'(0)=0$.  If
$\lsp(c_\tau)=\lsp(c_0)$ for all $\tau\in(-\eps,\eps)$, then
$c_\tau=c_0$ for all $\tau\in(-\eps,\eps)$.
\end{theorem}

\begin{theorem}
\label{thm:lsp-rig-reflecting}
Let $B=\bar B(0,1)\setminus B(0,R)\subset\R^n$, $n\geq2$ and $R>0$, be
an annulus.  Fix $\eps>0$ and let $c_\tau$, $\tau\in(-\eps,\eps)$, be
a $C^{1,1}$ function $[R,1]\to(0,\infty)$ satisfying the Herglotz
condition and the countable conjugacy condition and depending
$C^1$-smoothly on the parameter $\tau$.  If $\lsp'(c_\tau)=\lsp'(c_0)$
for all $\tau\in(-\eps,\eps)$, then $c_\tau=c_0$ for all
$\tau\in(-\eps,\eps)$.
\end{theorem}

Notice that dimension is irrelevant for the statements; if the sound
speed is fixed, the length spectrum is independent of dimension.

The following theorem states that the same rigidity result is true for
any finite disjoint union of manifolds of the types given in
theorems~\ref{thm:lsp-rig-diving} and~\ref{thm:lsp-rig-reflecting}.

\begin{theorem}
\label{thm:lsp-rig-combo}
Let $N$ and $N'$ be non-negative integers so that $N+N'\geq1$.
Let $R_1,R_2,\dots,R_N\in[0,1)$ and $R_1',R_2',\dots,R_{N'}'\in(0,1)$ be any
numbers.
Let $n_1,n_2,\dots,n_N$ and $n_1',n_2',\dots,n_{N'}'$ be integers, each of them at
least $2$.
Fix $\eps>0$ and let $c_{i,\tau}\colon[R_i,1]\to(0,\infty)$ for $i=1,\dots,N$ and
$c'_{i,\tau}\colon[R_i',1]\to(0,\infty)$ for $i=1,\dots,N'$ be $C^{1,1}$ functions
satisfying the Herglotz condition and the countable conjugacy condition and
depending $C^1$-smoothly on the parameter $\tau$.
For every $i$ such that $R_i=0$, we assume
$\left.\Der{r}c_{i,\tau}(r)\right|_{r=0}=0$.
If the set
\begin{equation}
\bigcup_{i=1}^N\lsp(c_{i,\tau})
\cup
\bigcup_{i=1}^{N'}\lsp'(c'_{i,\tau})
\end{equation}
is the same for all $\tau\in(-\eps,\eps)$, then every sound speed
$c_{i,\tau}$ and $c'_{i,\tau}$ is independent of the parameter $\tau$.
\end{theorem}

\begin{remark}
It does not matter whether for every periodic broken ray only the
primitive period is included in the length spectrum, or of its all
integer multiples.  The proofs of the three theorems above are the
same in both cases.
\end{remark}

Even more might be true, and we propose the following conjecture.

\begin{conjecture}
Under certain geometric hypotheses, a spherically symmetric manifold
is uniquely determined by its length spectrum.
\end{conjecture}

A verification of the conjecture would imply that such a
manifold is uniquely determined by its spectrum under some geometric assumptions.


\begin{remark}
Theorem~\ref{thm:lsp-rig-combo} may seem like an unnecessary
generalization of the two preceding theorems, but it has geophysical
significance. Consider a spherically symmetric model of the earth. It
essentially consists of three different parts, an inner ball and two
nested annuli. The full length spectrum of the Earth is the set of all
periodic orbits for the different polarized waves. The statement in
theorem~\ref{thm:lsp-rig-combo} (with $N=2$ and $N'=1$) is, however,
incomplete in the sense that the coupling between different
polarizations and transmission at boundaries are ignored.
However, this is the best toy model for which rigidity is currently known.
\end{remark}

Radial symmetry is an excellent approximation of the earth or planets
in general. This symmetry is not exact, and unfortunately our method
requires precise symmetry.  Assuming radial symmetry is not merely a
matter of technical convenience, but a truly necessary assumption.  A
key ingredient in the proof is that many periodic broken rays are
stable under deformations of the metric.  In spherical symmetry the
broken rays are only stable under deformations that preserve the
symmetry, otherwise they are typically unstable and the proof falls
apart.

The countable conjugacy condition should be a generic property of
sound speeds.  It is a technical assumption we do not like to make,
but it does not hinder the relevance for our planetary model.

The Herglotz condition is crucial for the geometry of the problem.
Without it the manifold would trap some geodesics inside and the
geometry of broken rays would be very different.  In the commonly used
Preliminary Reference Earth Model (PREM) both pressure and shear wave
speeds satisfy the Herglotz condition piecewise.  Due to the layered
structure of the Earth both wave speeds have jumps, whereas our result
assumes $C^{1,1}$ regularity.

\subsubsection{No inner reflections}

In this subsection we prove theorem~\ref{thm:lsp-rig-diving}.  All the
lemmas here are stated under the assumptions of the theorem. By
decreasing $\eps$ slightly we can assume that $c_\tau(1)$ is bounded
away from zero and infinity without any loss of generality.

Let $P_\tau\subset(R,1)$ be the set of radii for which the
corresponding broken ray -- which is unique up to rotations -- is
periodic with respect to the sound speed $c_\tau$.  A priori $P_\tau$
depends on $\tau$.  If $r\in P_\tau$, the corresponding broken ray has
$n_\tau(r)$ reflections and winds around the origin $m_\tau(r)$ times.
We choose all periodic broken rays to have minimal period, so the
natural numbers $n_\tau(r)$ and $m_\tau(r)$ are coprime.  If
$\alpha_\tau(r)$ is the angle we defined earlier (previously without
the dependence on $\tau$), we have the identity
\begin{equation}
\label{eq:pm=na}
\pi m_\tau(r)=n_\tau(r)\alpha_\tau(r).
\end{equation}

We denote the length of the periodic broken ray with radius $r$ by
$\ell_\tau(r)$. Simple geometrical considerations show that
$\ell_\tau(r)=2n_\tau(r)L_\tau(r)$, where $L_\tau(r)$ is defined like
$L(r)$ in~\eqref{eq:L}. We denote $\rho_\tau(r)=r/c_\tau(r)$. The Herglotz
condition states that $\rho_\tau'(r)>0$.

\begin{lemma}
\label{lma:uniform-bounds}
Assume $R>0$.
There is a constant $C>1$ so that
\begin{itemize}
\item $\frac1C<c_\tau(r)<C$,
\item $\frac1C(s-r)<\rho_\tau(s)-\rho_\tau(r)<C(s-r)$,
\item $\alpha_\tau(r)<C$ and
\item $L_\tau(r)<C$
\end{itemize}
for all $\tau\in(-\eps,\eps)$ and $r,s\in(R,1)$ with $s>r$.
\end{lemma}

\begin{proof}
It follows from the Herglotz condition that $r/c_\tau(r)\leq 1/c_\tau(1)$, whence
$c_\tau(r)\geq Rc_\tau(1)$.
Since $c_\tau(1)$ is uniformly bounded from below, the functions $c_\tau$ are all
uniformly bounded from below.
We assumed the sound speeds to be uniformly bounded from above in the theorem, so
for some constant $C>1$ we have $\frac1C<c_\tau(r)<C$ for all $\tau$ and $r$.

We write $LHS\lesssim RHS$ if there is a constant $C$ independent of $\tau$,
$r$ and $s$ so that $LHS<C\cdot RHS$.
By $LHS\approx RHS$ we mean $LHS\lesssim RHS\lesssim LHS$.

Assume $s>r$.
Since $c_\tau$ is $C^{1,1}$ and satisfies the Herglotz condition, we have
$0<\rho_\tau(s)-\rho_\tau(r)\lesssim s-r$.
On the other hand the minimum of $\rho_\tau'$ depends continuously on $\tau$ and is
always positive, so we have also $s-r\lesssim\rho_\tau(s)-\rho_\tau(r)$.
Thus $\rho_\tau(s)-\rho_\tau(r)\approx s-r$.

It follows from the previous observation that
\begin{equation}
\left(1-\left(\frac{\rho_\tau(r)}{\rho_\tau(s)}\right)^2\right)^{-1/2}
\approx
(s-r)^{-1/2}.
\end{equation}
Combining this with equations~\eqref{eq:alpha} and~\eqref{eq:L} gives the desired
uniform estimates for $\alpha_\tau$ and $L_\tau$.
\end{proof}

\begin{lemma}
\label{lma:R=0}
If $R=0$, then $\lim_{r\to0}\alpha_\tau(r)=\frac\pi2$ and
$\lim_{r\to0}L_\tau(r)=\int_0^1c_\tau(r)^{-1}\der r$.
\end{lemma}

\begin{proof}
In the limit the corresponding maximal geodesic tends to a diameter of the ball.
This limit is a radial geodesic and it is geometrically obvious that the angle and
length corresponding to it are the limits stated above.
\end{proof}

\begin{lemma}
\label{lma:lsp-countable}
The length spectrum $\lsp(c_\tau)$ is countable.
\end{lemma}

\begin{proof}
This follows from countability of $P_\tau$ given by
lemma~\ref{lma:ccc}, since $\lsp(c_\tau)=\{\ell_\tau(r);r\in
P_\tau\}$.
\end{proof}

We denote by $S_\tau$ the set of radii $r\in(R,1)$ for which
$\alpha(r)\in\pi\Q$ and $\alpha'(r)\neq0$.  Radii in this set will
correspond to stable broken rays as we shall see in
lemma~\ref{lma:stable-br} below.

\begin{lemma}
\label{lma:S-dense}
The set $S_\tau$ is countable and dense in $(R,1)$.
\end{lemma}

\begin{proof}
Let $C_\tau$ denote the set of zeros of $\alpha_\tau'$.
(We used the same notation in lemma~\ref{lma:ncp-implies-density}.)
By lemma~\ref{lma:ccc} we know that $P_\tau$ is dense and countable and $C_\tau$ is
closed and countable.
Therefore $S_\tau=P_\tau\setminus C_\tau$ is dense and countable.
\end{proof}

The next lemma shows that periodic broken rays are stable under variations of the
radial sound speed.
Periodic broken rays on a highly symmetric manifold are typically not stable under
all variations of the metric, but we are only looking at variations that preserve
the symmetry.

\begin{lemma}
\label{lma:stable-br}
For every $r\in S_0$ there is $\delta\in(0,\eps)$ a $C^1$ function
$\phi\colon(-\delta,\delta)\to(R,1)$ so that
\begin{itemize}
\item $\phi(0)=r$,
\item $\alpha_\tau(\phi(\tau))=\alpha_0(r)$ (and thus $\phi(\tau)\in P_\tau$),
\item $n_\tau(\phi(\tau))=n_0(r)$ and
\item $m_\tau(\phi(\tau))=m_0(r)$
\end{itemize}
for all $\tau\in(\delta,\delta)$.
In particular, $\tau\mapsto\ell_\tau(\phi(\tau))$ is differentiable.
\end{lemma}

\begin{proof}
Fix any $r\in S_0$.
We know that $\alpha_0'(r)\neq0$, so by the implicit function theorem there is a
$C^1$ function $\phi$ defined near zero so that
$\alpha_\tau(\phi(\tau))=\alpha_0(r)$ and $\phi(0)=r$.
Since $\alpha_\tau(\phi(\tau))$ is independent of $\tau$, so are the numbers
$n_\tau(\phi(\tau))$ and $m_\tau(\phi(\tau))$.
Differentiability of the length follows from the fact that the reflection number is
constant and $L$ is differentiable.
\end{proof}

\begin{lemma}
\label{lma:length-to-pbrt}
Let $\phi\colon(-\delta,\delta)\to(R,1)$ for any $\delta>0$ be a $C^1$ function
satisfying $\phi(\tau)\in P_\tau$ for all $\tau\in(-\delta,\delta)$.
Let $\gamma_0\colon[0,T]\to M$ be a periodic broken ray with radius $\phi(0)$.
Then
\begin{equation}
2
\left.\Der{\tau}\ell_\tau(\phi(\tau))\right|_{\tau=0}
=
\int_0^T\left.\Der{\tau}c^{-2}_\tau(\gamma_0(t))\right|_{\tau=0}\der t.
\end{equation}
\end{lemma}

\begin{proof}
A version of this lemma with non-periodic broken rays of finite length
and arbitrary variations of the metric was given
in~\cite[theorem~17]{IS:brt-pde-1obst}. (There is an error in the
formula of the cited theorem; the factor $2$ should be replaced with
$\frac12$.)  Applying that result for conformal variations and broken
rays that have the same starting and ending points gives the desired
claim.
\end{proof}

\begin{lemma}
\label{lma:pbrt-to-abel}
If $f\colon M\to\R$ is a continuous radially symmetric function (identified as a
function $f\colon(R,1]\to\R$) and $r\in P_0$, then the integral of $f$ over any
periodic geodesic (with respect to sound speed $c_0$) of radius $r$ is
\begin{equation}
\label{eq:abel-transform}
2n_0(r)
\int_r^1\frac{f(s)}{c(s)}\left(1-\left(\frac{rc(s)}{sc(r)}\right)^2\right)^{-1/2}\der
s.
\end{equation}
\end{lemma}

\begin{proof}
The proof is a simple calculation, similar to the one leading to
equation~\eqref{eq:L}.
\end{proof}

\begin{lemma}
\label{lma:abel}
If $A f(r)$ is the function of~\eqref{eq:abel-transform}, then the map
$f\mapsto Af$ takes continuous functions to continuous functions and is injective on the space of continuous functions.
\end{lemma}

\begin{proof}
This follows from theorems 5 and 12 (or lemma 25) of~\cite{dHI:radial-xrt}.
\end{proof}

In the $C^\infty$ setting relevant for the spectrum of the
Laplace--Beltrami operator this was shown by
Sharafutdinov~\cite{Sharaf97}. With all these lemmas as ingredients,
it is difficult not to prove theorem~\ref{thm:lsp-rig-diving}.

\begin{proof}[Proof of theorem~\ref{thm:lsp-rig-diving}]
Take any radius $r\in S_0$ and let $\phi$ be the function given by
lemma~\ref{lma:stable-br}. We know that
$\tau\mapsto\ell_\tau(\phi(\tau))$ is differentiable
(lemma~\ref{lma:stable-br}),
$\ell_\tau(\phi(\tau))\in\lsp(c_\tau)=\lsp(c_0)$ and $\lsp(c_0)$ has
no interior (lemma~\ref{lma:lsp-countable}).  Therefore this function
is constant.

By lemma~\ref{lma:length-to-pbrt} this implies that the variation of the wave
speed, $f=\left.\Der{\tau}c^{-2}_\tau\right|_{\tau=0}\colon M\to\R$, integrates to zero
over all periodic geodesics of radius $r$.
This function $f$ is radially symmetric, so we can think of it as a function
$(R,1]\to\R$.
By lemma~\ref{lma:pbrt-to-abel} we know that
\begin{equation}
\label{eq:abel=0}
\int_r^1\frac{f(s)}{c(s)}\left(1-\left(\frac{rc(s)}{sc(r)}\right)^2\right)^{-1/2}\der
s
=
0.
\end{equation}
Equation~\eqref{eq:abel=0} is true for a dense set of radii
$r\in(R,1)$ by lemma~\ref{lma:S-dense}, so it follows from
lemma~\ref{lma:abel} that in fact $f$ vanishes identically.

We have found that $\Der{\tau}c_\tau=0$ at $\tau=0$.
The choice $\tau=0$ was in no way important to this argument, so in fact
$\Der{\tau}c_\tau=0$ for all $\tau\in(-\eps,\eps)$.
This means that all sound speeds $c_\tau$ indeed coincide.
\end{proof}

A key step in the reasoning in the preceding proof can be stated as
follows: A radially symmetric function is uniquely determined by its
integrals over all periodic broken rays. There is even a
reconstruction formula for this problem:
\cite[Remark~27]{dHI:radial-xrt}. Radial symmetry is important, as a
general smooth function is not uniquely determined. Only the even
part of the function can be recovered efficiently. This is in sharp
contrast to the case of geodesic X-ray tomography on such manifolds
where the ray transform has no kernel. See~\cite{dHI:radial-xrt} for
details.

\subsubsection{Inner reflections included}

In this subsection we will prove theorems~\ref{thm:lsp-rig-reflecting}
and~\ref{thm:lsp-rig-combo}.

The only ingredient we need to add to the proof of
theorem~\ref{thm:lsp-rig-diving} is the following lemma.  We will
prove the lemma after showing how it completes the proofs of the
theorems.

\begin{lemma}
\label{lma:reflecting-spectrum}
Let $M=\bar B(0,1)\setminus B(0,R)\subset\R^n$, $R>1$ and $n\geq2$, be
an annulus. Equip $M$ with a radially symmetric $C^{1,1}$ sound speed
satisfying the Herglotz condition.  Then the set of all lengths of
periodic broken rays that reflect on the inner boundary $r=R$ is
countable.
\end{lemma}

We note that no assumption was made on conjugate points in the lemma.

\begin{proof}[Proof of theorem~\ref{thm:lsp-rig-reflecting}]
We denote the set of all lengths of periodic broken rays with respect
to sound speed $c_\tau$ that reflect on the inner boundary by
$H_\tau$.  Then we have $\lsp'(c_\tau)=\lsp(c_\tau)\cup H_\tau$.  As
in the proof of theorem~\ref{thm:lsp-rig-diving}, we have a dense set
of radii $r\in(R,1)$ for which there is a family of corresponding
periodic broken rays varying continuously in $\tau$; this set was
denoted by $S_0$.  Since $\lsp'(c_\tau)$ is independent of $\tau$ and
each set $H_\tau$ is countable, the lengths of the periodic broken
rays in this family must in fact be independent of $\tau$.  The rest
of the proof can be concluded as that of
theorem~\ref{thm:lsp-rig-diving}.  We only need the integrals of the
variations of the sound speed over broken rays that do not hit the
inner boundary.
\end{proof}

\begin{proof}[Proof of theorem~\ref{thm:lsp-rig-combo}]
We denote the individual Riemannian manifolds in question by
$M_1,\dots,M_N$ and $M_1',\dots,M_{N'}'$, equipped with their
respective sound speeds.  Each one of them has a countable length
spectrum, so the length spectrum of the whole system is still
countable.  We can then use the argument presented in the proof of
theorem~\ref{thm:lsp-rig-reflecting} to conclude for each manifold
separately that the sound speed has to be independent of the parameter
$\tau$.
\end{proof}

\begin{proof}[Proof of lemma~\ref{lma:reflecting-spectrum}]
Consider a geodesic joining the two boundaries; a periodic broken ray
of the kind we need to study is a finite union of rotations and
reflections of such a geodesic. We parametrize this geodesic with arc
length starting from the inner boundary. Due to symmetry the geodesic
is confined to a two-dimensional plane, and in this plane we may use
polar coordinates.  In these coordinates the geodesic is $[0,T]\ni
t\mapsto(r(t),\theta(t))$, with $r(0)=R$ and $r(T)=1$. We denote
$\rho(r)=r/c(r)$.

Let the angle between the geodesic and the radially outward pointing
normal vector be $\omega$. (The metric is conformally Euclidean so we
may use the Riemannian metric or the Euclidean metric to measure
angles at a point.) We assume for the time being that
$\omega\in(0,\frac\pi2)$.  Since the geodesic has unit speed, we have
$\cos\omega=\rho(R)\theta'(0)$.  With this information it is easy to
see that the constant value of the angular momentum
$\rho(r(t))^2\theta'(t)$ is $\rho(R)\cos\omega$.

Using unit speed and the conservation of angular momentum one easily
finds that the change in the angular coordinate over the geodesic is
\begin{equation}
\theta(T)-\theta(0)
=
\int_R^1\frac{\rho(R)\cos\omega}{c(r)\rho(r)^2}\left(1-\left(\frac{\rho(R)\cos\omega}{\rho(r)}\right)^2\right)^{-1/2}\der
r.
\end{equation}
This angular difference depends on $\omega$, but it is most convenient
to think of it as a function of $z\coloneqq\rho(R)\cos\omega$ (the
angular momentum). We denote this angular difference by $\beta(z)$.
Its geometrical meaning is the same as that of $\alpha(r)$ defined
earlier, but we use a different letter to avoid confusion.

Since
\begin{equation}
\label{eq:reflection-angle-difference}
\beta(z)
=
\int_R^1\frac{z}{c(r)\rho(r)^2}\left(1-\left(\frac{z}{\rho(r)}\right)^2\right)^{-1/2}\der
r,
\end{equation}
an easy calculation gives
\begin{equation}
\beta'(z)
=
\int_R^1\frac{1}{c(r)\rho(r)^2}\left(1-\left(\frac{z}{\rho(r)}\right)^2\right)^{-3/2}\der
r.
\end{equation}
This derivative is positive, so $\beta$ is in fact a homeomorphism
from $[0,1]$ to its image; the limits at $0$ and $1$ can be checked to
be finite.

We may exclude radial geodesics. For non-radial geodesics the
integral of~\eqref{eq:reflection-angle-difference} is non-singular and
differentiation under the integral sign is simple. For the similar
result corresponding to the diving waves (the function $\alpha$), the
derivative is more complicated;
cf.~\cite[proposition~15]{dHI:radial-xrt}.

The broken ray corresponding to angular momentum $z$ is periodic if
and only if $\beta(z)\in\pi\Q$. Since $\beta$ is a homeomorphism
between intervals, the set of angular momenta $z$ corresponding to
periodicity is countable. Therefore the set of corresponding lengths
is also countable.
\end{proof}

\subsection{Spectral rigidity}
\label{sec:spec-rig-proof}

We are now ready to prove theorem~\ref{thm: intro spec rigidity}. The
proof is trivial at this point, but we record it explicitly for
completeness.

\begin{proof}[Proof of theorem~\ref{thm: intro spec rigidity}]
Let us first consider the simplest case where $c$ satisfies the periodic conjugacy condition.

The trace of the Green's function is determined by $\spec(c)$
through proposition~\ref{prop:Trace Formula}. Since the trace as a
function of $t \in \R$ is singular precisely at the length spectrum,
$\spec(c)$ determines $\lsp'(c)$. Then rigidity of the spectrum
$\spec(c)$ follows from that of $\lsp'(c)$; see
theorem~\ref{thm:lsp-rig-reflecting}.

Let us then drop the periodic conjugacy condition.
The Neumann spectrum $\spec(c)$ still determines the trace of the Green's function.
As pointed out in remark~\ref{rmk:tr-pcc-nondeg}, the singularities determine the part of the length spectrum corresponding to periodic broken rays satisfying the clean intersection property.
Under the Herglotz condition and the countable conjugacy condition a periodic broken ray of radius $r$ satisfies the clean intersection property if and only if $\alpha'(r)\neq0$; see remark~\ref{rmk:pcc}.
In the notation of lemma~\ref{lma:ncp-implies-density} and denoting the length of a geodesic of radius $r$ by $\ell(r)$, the problematic primitive lengths are precisely $\ell(r)$ for $r\in C\cap P$.
Since we assumed the length spectrum to be non-degenerate, none of the lengths $\ell(r)$ for $r\in P\setminus C$ coincide with the problematic ones.
Access to all radii $r\in P\setminus C$ is sufficient for the proof of length spectral rigidity.
Notice that $C$ is exactly the same set that corresponds to possibly unstable periodic broken rays, and this data was ignored in the proof of length spectral rigidity anyway.
\end{proof}

\appendix

\section{Debye Expansion}\label{sec:Debye exp}

Here, we describe how to relate the sum of eigenfunctions to a kernel
closely related to the propagator. First, we note that
\[
   \int_0^{\infty} {}_n f_l(t) e^{-\ii \omega t} \, \dd t
           = \frac{1}{\omega_n(k)^2 - \omega^2}
\]
(cf.~\eqref{eq:nflom}). We substitute, again, $k = \omega p$, and
introduce the Debye expansion for $p$ fixed. Thus we can isolate the
different regimes:
\begin{itemize}
\item
\textbf{Diving} ($\CMB/
c(\CMB) < p < \surf /
c(\surf)$: The dispersion relations satisfy,
\[
   \omega_n(p) \coloneqq \omega_n = \left( n + \frac{5}{4} \right) \pi
               \overline{\tau}(p)^{-1} ,
\quad
   \overline{\tau}(p)
       = \int_{\rturn}^{\surf}
                          \wkb(r';p) \dd r'
\]
(cf.~\eqref{eq:efII}) and the WKB eigenfunctions yield
\begin{multline}
   U_{n}(r;p) U_{n}(r_0;p) = 4 T^{-1}
      (r r_0)^{-1}
\\
      c(r)^{-1} \rho(r)^{-1/2} c(r_0)^{-1} \rho(r_0)^{-1/2}
      \wkb^{-1/2}(r;p) \wkb^{-1/2}(r_0;p)
\\
      \sin[(n+\tfrac{5}{4}) \pi \tau(r;p) / \overline{\tau}(p)]
      \sin[(n+\tfrac{5}{4}) \pi \tau(r_0;p) / \overline{\tau}(p)] ,
\end{multline}
where
\[
   \omega \widetilde{\tau}(r;p)
       = \omega \int_{\rturn}^{\surf} \wkb(r';p) \dd r' + \frac{\pi}{4},
\]
to leading order. In fact, $2 \overline{\tau} \, T^{-1}$ cancels
against $(1 - c_n^{-1} C_n)^{-1}$; hence, we premultiply $U_{n}(r;p)
U_{n}(r_0;p)$ by $\overline{\tau}$ before analyzing the
summation. The summation,
\[
      (r r_0)
      c(r) \rho(r)^{1/2} c(r_0) \rho(r_0)^{1/2}
      \wkb^{1/2}(r;p) \wkb^{1/2}(r_0;p)
      \frac{T}{4 \overline{\tau}} \,
      \sum_{n=0}^{\infty}
      \frac{U_{n}(r;p) U_{n}(r_0;p)}{\omega_n(p)^2 - \omega^2}
\]
can be carried out to yield
\begin{equation}
\begin{split}
&
   \frac{1}{\overline{\tau}} \sum_{n=0}^{\infty}
   \frac{\sin[(n+\tfrac{5}{4}) \pi \widetilde{\tau}(r;.) /
              \overline{\tau}]
      \sin[(n+\tfrac{5}{4}) \pi \widetilde{\tau}(r_0;.) /
           \overline{\tau}]}{
      [(n+\tfrac{5}{4}) \pi / \overline{\tau}]^2 - \omega^2}
\\&\quad
   = \frac{1}{8 \ii \omega}\
   \sum_{i=1}^{\infty} \exp\left[ -\ii \omega \tau_i(r,r_0;p)
              + \ii N_j \frac{\pi}{2} \right] ,
\end{split}
\end{equation}
where
\begin{eqnarray*}
   \tau_1(r,r_0;p) &=&
        \abs{ \int_{r_0}^r \wkb(r';p) \dd r' }
\\
   \tau_2(r,r_0;p) &=&
        \int_{\rturn}^{r_0} \wkb(r';p) \dd r'
      + \int_{\rturn}^r \wkb(r';p) \dd r' ,
\\
   \tau_3(r,r_0;p) &=&
        \int_{r_0}^{\surf}
                          \wkb(r';p) \dd r'
      + \int_r^{\surf}
                          \wkb(r';p) \dd r' ,
\\
   \tau_4(r,r_0;p) &=&
        \int_{\rturn}^{\surf}
                           \wkb(r';p) \dd r'
      - \abs{ \int_{r_0}^r \wkb(r';p) \dd r' } ,
\\
   \tau_i(r,r_0;p) &=& \tau_{i-4}(r,r_0;p) +
        2 \int_{\rturn}^{\surf}
                           \wkb(r';p) \dd r' ,\quad
   i = 5,6,\ldots
\end{eqnarray*}
while
\begin{eqnarray*}
   N_1 &=& 0 ,
\\
   N_2 &=& 1 ,
\\
   N_3 &=& 0 ,
\\
   N_4 &=& 1 ,
\\
   N_i &=& N_{i-4} + 1 ,\quad
   i = 5,6,\ldots
\end{eqnarray*}
\item
\textbf{Reflecting} ($0 < p < \CMB /
c(\CMB)$: The dispersion relations satisfy,
\[
   \omega_n(p) \coloneqq \omega_n = (n+1) \pi \overline{\tau}(p)^{-1} ,
\quad
   \overline{\tau}(p)
       = \int_{\CMB}^{\surf}
                          \wkb(r';p) \dd r'
\]
(cf.~\eqref{eq:efIII}) and the WKB eigenfunctions yield
\begin{multline}
   U_{n}(r;p) U_{n}(r_0;p) = 4 T^{-1}
      (r r_0)^{-1}
\\
      c(r)^{-1} \rho(r)^{-1/2} c(r_0)^{-1} \rho(r_0)^{-1/2}
      \wkb^{-1/2}(r;p) \wkb^{-1/2}(r_0;p)
\\
      \sin[(n+1) \pi \widetilde{\tau}(r;p) / \overline{\tau}(p)]
      \sin[(n+1) \pi \widetilde{\tau}(r_0;p) / \overline{\tau}(p)] ,
\end{multline}
where
\[
   \omega \widetilde{\tau}(r;p)
       = \omega \int_{\CMB}^r
                            \wkb(r';p) \dd r' + \frac{\pi}{2} ,
\]
to leading order. In fact, $2 \overline{\tau} \, T^{-1}$ cancels
against $(1 - c_n^{-1} C_n)^{-1}$; hence, we premultiply $U_{2;n}(r;p)
U_{2;n}(r_0;p)$ by $\overline{\tau}$ before analyzing the
summation. The summation,
\[
      (r r_0)
      c(r) \rho(r)^{1/2} c(r_0) \rho(r_0)^{1/2}
      \wkb^{1/2}(r;p) \wkb^{1/2}(r_0;p)
      \frac{T}{4 \overline{\tau}} \,
        \sum_{n=0}^{\infty} \frac{U_{n}(r;p)
                U_{n}(r_0;p)}{\omega_n(p)^2 - \omega^2}
\]
can be carried out to yield
\begin{multline}
   \frac{1}{\overline{\tau}} \sum_{n=0}^{\infty}
   \frac{\sin[(n+1) \pi \widetilde{\tau}(r;.) / \overline{\tau}]
      \sin[(n+1) \pi \widetilde{\tau}(r_0;.) / \overline{\tau}]}{
      [(n+1) \pi / \overline{\tau}]^2 - \omega^2}
\\
   = \frac{1}{4 \omega}\
   \frac{\cos[\omega (\overline{\tau} - (\widetilde{\tau}(r;.) +
   \widetilde{\tau}(r_0;.))] -
         \cos[\omega (\overline{\tau} - (\widetilde{\tau}(r;.) -
         \widetilde{\tau}(r_0;.))]}{
      \sin(\omega \overline{\tau})} .
\end{multline}
Upon representing the cosines and sine on the right-hand side in terms
of complex exponentials, and then applying the binomial expansion, we
obtain
\begin{equation}
   \frac{1}{\overline{\tau}} \sum_{n=0}^{\infty}
   \frac{\sin[(n+1) \pi \widetilde{\tau}(r;.) / \overline{\tau}]
      \sin[(n+1) \pi \widetilde{\tau}(r_0;.) / \overline{\tau}]}{
      [(n+1) \pi / \overline{\tau}]^2 - \omega^2}
   = \frac{1}{8 \ii \omega}\
   \sum_{i=1}^{\infty} \exp[-\ii \omega \tau_i(r,r_0;p)] ,
\end{equation}
where
\begin{eqnarray*}
   \tau_1(r,r_0;p) &=&
        \abs{ \int_{r_0}^r \wkb(r';p) \dd r' } ,
\\
   \tau_2(r,r_0;p) &=&
        \int_{\CMB}^{r_0}
                           \wkb(r';p) \dd r'
      + \int_{\CMB}^r
                           \wkb(r';p) \dd r' ,
\\
   \tau_3(r,r_0;p) &=&
        \int_{r_0}^{\surf}
                          \wkb(r';p) \dd r'
      + \int_r^{\surf}
                          \wkb(r';p) \dd r' ,
\\
   \tau_4(r,r_0;p) &=&
        \int_{\CMB}^{\surf}
                           \wkb(r';p) \dd r'
      - \abs{ \int_{r_0}^r \wkb(r';p) \dd r' } ,
\\
   \tau_i(r,r_0;p) &=& \tau_{i-4}(r,r_0;p) +
        2 \int_{\CMB}^{\surf}
                           \wkb(r';p) \dd r' ,\quad
   i = 5,6,\ldots
\end{eqnarray*}
To unify the notation, we can introduce $N_i = 0$, $i = 1,2,\ldots$
\end{itemize}

\section{General framework for symmetries in a manifold}
\label{sec:Trace with symmetry}

In this appendix, we consider more general situations for trace
formulas when the manifold has symmetries given by a compact Lie group
and we show how the quantities appearing in proposition
\ref{prop:Trace Formula} are special cases of a general framework. Our
constructions here are inspired by the work in
\cite{Creagh91,Creagh92,Uribe95}. In particular, we provide general
representation of the critical manifolds $C_T$ described in sections
\ref{sec: efunction and quant} and~\ref{sec: Melrose/guill
  construction} under symmetry from a compact Lie group $G$, which
corresponds to the principal symbol of $\Delta_g$ possessing
symmetries. See the works of Cassanas~\cite{Cassanas06} and Gornet
\cite{Gornet} for other closely related work.

Let $(M,g)$ be a compact Riemannian manifold with boundary satisfying
the same assumptions as in~\cite{Mel79}. In fact, all that is
necessary is that one may construct a parametrix in the same form as
\cite{Mel79} for the Neumann wave propagator. We assume that a compact
Lie group $G$ has a symplectic group action on $T^*M$ and it accounts
for all the symmetries of the Hamiltonian $p(x,\xi)= \abs{\xi}^2_g$,
which is the principal symbol of $\Delta_g$. That is, $p(g.z) = p(z)$
for each $z \in T^*M$ and $g \in G$, where $g.z$ or $gz$ denotes the
group action. The assumption is merely that $G$ accounts for all such
symmetries.

The symmetry in the Hamiltonian implies that the fixed point manifolds
$C_T$ introduced in section~\ref{sec: efunction and quant} are
in fact invariant under the group action: $g.C_T \subset C_T$ for
each $g \in G$. The main assumption we make is that the connected
components of $C_T$ are the $G$-orbit of a particular periodic
bicharacteristic $\gamma$ and it has no other symmetry; i.e. assuming
$C_T$ is already connected, $C_T = G\gamma \coloneqq \{ g . x; g \in G
\text{ and }x \in \Im(\gamma)\}.$ We write $[\gamma]$ to denote
this set, which is also a certain equivalence class of $\gamma$ as
described in section~\ref{sec: efunction and quant}.
We note that this assumption captures generic situations since on general manifolds without symmetry, $C_T$ has dimension $1$ (see~\cite[p. 61]{DG75}), and an increase in dimension should only come from a group symmetry.

Next, we set $L = I - \der\Phi^T(m)$ for $m \in C_T$. In the context of
group symmetries, the map $I-\der\Phi^T\colon
T_mS^*M/\ker(L) \to T_m S^* M / \ker(L)$ is in
general not an isomorphism anymore~\cite[lemma
  4.4]{DG75})\footnote{This map happened to be an isomorphism for the
  $SO(n)$ group in dimension $3$, but this will normally not happen in
  higher dimensions~\cite{Creagh91,Creagh92}.}. However, the map
\[
I - \der\Phi^T\colon T_mT^*M/\ker(L^2) \to T_mT^*M/\ker(L^2)
\]
will generically be an isomorphism~\cite[appendix
  A.1]{Uribe95}, \cite{Creagh91,Creagh92,Cassanas06}.  The above map
is what we will now refer to as the Poincar\'{e} map, which we denote
$I-P_{[\gamma]}$ as described in section~\ref{sec: efunction and
  quant}. We also assume the \emph{clean intersection hypothesis}
described in~\cite{DG75,Mel79} so that $C_T$ is a submanifold and
$\ker(L) = T_mC_T$.
See remark~\ref{rmk:pcc} for a geometric description of the clean intersection hypothesis in spherical symmetry.

The new part of the trace formula will come from
$\ker(L^2) / \ker(L)$ whose presence is
already seen in the calculation of the normal Hessian for the
stationary phase analysis in section~\ref{sec: Melrose/guill
  construction}. Based on the arguments in~\cite{Creagh91, Creagh92},
there is a geometric scalar quantity, which we denote by
$d_{[\gamma]}$, defined on $C_T$, that is nonvanishing and associated
with the action of $I-\der\Phi^T$ on $\ker(L^2)/ \ker(L)$
\cite[appendix]{Creagh91}. This is a quantity directly related to
periodic orbits in the reduced phase space formally written as
$T^*M/G$ and defined in~\cite[Chapter 4]{Marsden}. Under certain
geometric conditions, $\abs{I - P_{[\gamma]}}$ and $d_{[\gamma]}$ will
stay constant over $C_T = [\gamma]$. For our $SO(3)$ action,
$\ker(L^2)/\ker(L)$ was only one-dimensional determined by the
infinitesimal generator of dilations in the dual
variables. Geometrically, this corresponds to merely increasing the
speed of a particular geodesic. Analytically, for a geodesic
$\gamma_\nu$ starting at $\nu \in S^*M$, it corresponds to the Jacobi
field $J(t) = t\dot\gamma_\nu(t)$ (see~\cite[p. 70]{DG75} for more
details). Hence, in our $SO(3)$ case, $T_mT^*M/\ker(L^2)$ is
indeed isomorphic to $T_mS^*M/T_mC_T$, but generally,
$\ker(L^2)/ \ker(L)$ will be much bigger when there is
symmetry.

Next, our assumptions imply that there is a natural surjective map
$[0,T_\gamma^\sharp) \times G \to C_T$ which usually fails to be
  injective, where $T^\sharp_\gamma$ is the primitive period of
  $\gamma$. For example, in the analysis in section~\ref{sec: trace
    formula}, one does not need the entire rotation group $SO(3)$ to
  form $C_T$, which would in fact create multiple copies, since
  certain elements of $SO(3)$ already map $\gamma$ to itself. If the
  map is locally injective, then following~\cite[Section 3]{Creagh92},
  we assume there is a discrete subgroup $I(\gamma)$ of $G$ that
  carries $\gamma$ to itself, and we denote $N_\gamma$ to be the
  number of elements in $I(\gamma)$, corresponding to the analogous
  quantity appearing in proposition~\ref{prop:Trace Formula}.

In summary, we have accounted for all the pieces in our trace formula (proposition~\ref{prop:Trace Formula}) in a general Lie group framework so that our formula may be compared to the formula appearing in~\cite{Creagh92} albeit in a different setting.

\section{A remark about spherical symmetry}
\label{app:symmetry}

Any spherically symmetric manifold is in fact of the form we consider
-- radially conformally Euclidean.

\begin{proposition}
\label{prop:rce}
Let $A\subset\mathbb R^n$ be an annulus (difference of two cocentric
balls) and $g\in C^k$ a rotationally symmetric metric in the sense
that $g=U^*g$ for all $U\in SO(n)$.  Then $(A,g)$ is isometric to a
Euclidean annulus with the Euclidean metric multiplied with a
conformal factor $c^{-2}(\abs{x})$ with $c\in C^k$.
\end{proposition}

\begin{proof}
In this proof it is more convenient to write $c^{-2}=\eta$.
This notation is not used elsewhere.

By rotational symmetry it suffices to consider the metric at points
$z_r=(r,0,\dots,0)$ for $r\in[R,1]$.
Near $z_r$ we take local coordinates $(x,y)\in\R\times\R^{n-1}$ such that $(x,y)$
represents the point $(r+x,y)$ on $A$.
In these coordinates we can write the metric at $z_r$ as the matrix
\begin{equation}
\label{eq:metric1}
g_{z_r}
=
\begin{pmatrix}
a(r)&b(r)^T\\
b(r)&C(r)
\end{pmatrix},
\end{equation}
where $a(r)$, $b(r)$ and $C(r)$ are a number, a vector and a matrix
depending on $r$.  We will drop the argument $(r)$ where it is
implicitly clear.  We will first show that $b$ is identically zero (if
$n\geq3$) or becomes zero after applying a diffeomorphism that
preserves rotational symmetry ($n=2$) and $C(r)=c(r)I$ for some scalar
function $c$ (trivial for $n=2$).

We first consider the case $n\geq3$. For any $R\in SO(n-1)$ we
have $U_R\coloneqq 1\oplus R\in SO(n)$.  Since $g_{z_r}$ must be
invariant under $U_R$, we have $Rb=b$ and $RCR^{-1}=C$.  But this
holds for all $R\in SO(n-1)$, so $b=0$ and $C$ is a multiple of
identity.

We then turn to the case $n=2$. Now $b$ and $C$ are scalars, and we
write $C=c$.  By positive definiteness of the metric we have $a>0$,
$c>0$, and $b^2<ac$.

We write points on $A$ in polar coordinates $(r,\theta)$ and define a
function $F_\phi\colon A\to A$ parametrized by a function $\phi\colon
[R,1]\to\R$ by setting $F(r,\theta)=(r,\theta+\phi(r))$.  If $\phi$ is
$C^k$, then $F$ is clearly a $C^k$ diffeomorphism.  After we fix
$r_0\in[R,1]$, we may assume that $\phi(r_0)=0$ by rotational
symmetry, so that $F(z_{r_0})=z_{r_0}$.  In the Euclidean coordinates
$(x,y)$ near $z_{r_0}$ we have
\begin{equation}
DF_{z_{r_0}}
=
\begin{pmatrix}
1&0\\
\alpha(r_0)&1
\end{pmatrix}
\end{equation}
where $\alpha(r)=r\phi'(r)$, which implies
\begin{equation}
\label{eq:metric2}
F_\phi^*g_{z_{r_0}}
=
\begin{pmatrix}
a+2b\alpha+c\alpha^2&b+c\alpha\\
b+c\alpha&c
\end{pmatrix}.
\end{equation}
If we choose the function $\phi$ so that $r\phi'(r)=-b(r)/c(r)$, the
metric~\eqref{eq:metric2} becomes diagonal.
Since $c$ is bounded from below uniformly on $[R,1]$, the function $\phi(r)=-\int^r
b(s)/sc(s)\der s$ is well defined and $C^{k+1}$ if the original metric is $C^k$;
additive constants are irrelevant, since they correspond to rotations of the entire
annulus.

We have now shown that the metric can be assumed to have the form
\begin{equation}
\label{eq:metric3}
g_{z_r}
=
\begin{pmatrix}
a(r)&0\\
0&c(r)I
\end{pmatrix}.
\end{equation}
If $\rho\colon [R,1]\to\R$ is a strictly increasing $C^1$ function, we define the
change of variable (again in polar coordinates)
$G_\rho(r,\theta)=(\rho(r),\theta)$.
The function $\rho$ is a diffeomorphism and we denote $\sigma=\rho^{-1}$.
We will later choose $\rho$ so that $\rho(1)=1$ and $\rho(R)>0$, which makes
$G_\rho\colon A\to \bar B(0,1)\setminus B(0,\rho(R))$ a diffeomorphism.

A simple calculation shows that
\begin{equation}
\label{eq:metric4}
(G^{-1}_\rho)^*g_{z_r}
=
\begin{pmatrix}
a(\sigma(r))\sigma'(r)^2&0\\
0&c(\sigma(r))(\sigma(r)/r)^2I
\end{pmatrix}.
\end{equation}
If we construct $\rho$ so that
\begin{equation}
\label{eq:cond1}
a(\sigma(r))\sigma'(r)^2=c(\sigma(r))(\sigma(r)/r)^2,
\end{equation}
the metric~\eqref{eq:metric4} is a multiple of the identity matrix (conformally Euclidean) and
\begin{equation}
(\bar B(0,1)\setminus B(0,\tau),(G^{-1}_\rho)^*g)
\end{equation}
is a manifold of the desired form.
We will see that $\rho\in C^{k+1}$, and this shows the regularity claim.

For convenience, we change variable from $r$ to $s=\sigma(r)$.
Condition~\eqref{eq:cond1} now becomes
\begin{equation}
\Der{s}\log\rho(s)
=
\frac{1}{s}\sqrt{\frac{a(s)}{c(s)}}.
\end{equation}
We thus choose
\begin{equation}
\label{eq:rho-set}
\rho(s)
=
\exp\left(
\int_1^s\frac{1}{t}\sqrt{\frac{a(t)}{c(t)}}\der t
\right).
\end{equation}
Since the integrand is strictly positive and $C^k$, the function $\rho$ is a strictly increasing
$C^{k+1}$ function as claimed.
We also claimed earlier that $\rho$ satisfies $\rho(1)=1$ and $\rho(R)>0$, and
these properties can be read in the representation~\eqref{eq:rho-set}.
\end{proof}

\begin{remark}
The diffeomorphism of proposition~\ref{prop:rce} is in fact radial if $n\geq3$.
It will also be necessarily radial in two dimensions if the metric is invariant under the action of $SO(2)$, but also that of $O(2)$.
\end{remark}

\section{Some exotic spherically symmetric geometries}
\label{app:edge-cases}

Our results assumed several geometric hypotheses.
In this appendix we explore some problematic spherically symmetric geometries which are ruled out by our assumptions.

First, recall that countability of the length spectrum was proven in lemma~\ref{lma:lsp-countable} under the Herglotz and countable conjugacy conditions.
It is not known whether the length spectrum can be uncountable without these assumptions.

Let us then see concrete examples where our assumptions are violated and it leads to problematic behavior:

\begin{example}
If the derivative $\frac{\der}{\der r}\left(\frac{r}{c(r)}\right)$
vanishes in an open set of radii, then that part of the manifold is
isometric to a cylinder $S^2(0,a)\times(0,b)$ for some $a,b>0$. The
great circles of $S^2$ with the second variable constant are periodic
geodesics, and they all have the same length $T=2\pi a$. The dimension of the
corresponding fixed point set $C_T$ is $3$. This manifold is trapping
and the periodic orbits in question do not reach the boundary.
The Herglotz condition is violated.
\end{example}

\begin{example}
Consider the closed hemisphere of $S^3$ or any other $S^n$, $n\geq2$.
Now all broken rays are periodic, and all primitive periods are
$2\pi$. The primitive length spectrum is degenerated into a single
point. The Herglotz condition fails at the boundary (but only there),
and the endpoints of \emph{all} maximal geodesics are conjugate. The
fixed point set has full dimension; for $n=3$ we have
$\dim(C_{2\pi})=5$. This manifold is non-trapping.
\end{example}

\begin{example}
Consider the closed hemisphere $H\subset S^2$. In polar coordinates we
can identify $H$ near the boundary (equator) with
$(R,1]\times[0,2\pi]$, where the angles $0$ and $2\pi$ are identified
in the obvious way. For any $L\in(0,2\pi)$, we can take the
submanifold $(R,1]\times[0,L]$ and identify angles $0$ and $L$. If
$L\notin\pi\Q$, then no broken ray near the boundary is
periodic. By angular rescaling we may write this as a rotation
invariant metric near the boundary of the closed disc $\bar D$ and
continue it to a rotation invariant metric in the whole disc. Now
there is an open set of points on the sphere bundle containing no
periodic orbits, and this open set can be made large. Only rays going
close enough to the center of the disc will contribute to the length
spectrum. This manifold is also non-trapping, the Herglotz condition
fails at the boundary, and the countable conjugacy condition is
violated.
\end{example}

\bibliography{poisson_summation}
\bibliographystyle{plain}

\end{document}